\documentclass[11pt]{article}

\usepackage{amssymb,amsthm,amsmath,hyperref,txfonts,mathrsfs}

\usepackage{color}

\usepackage{graphicx,float}

\numberwithin{equation}{section}

\newtheorem{thm}{Theorem}[section]
\newtheorem{lem}{Lemma}[section]

\newtheorem{prop}{Proposition}[section]

\newcommand{\beq}{\begin{eqnarray}}
\newcommand{\eeq}{\end{eqnarray}}
\newcommand{\beqno}{\begin{eqnarray*}}
\newcommand{\eeqno}{\end{eqnarray*}}
\newcommand{\be}{\begin{equation}}
\newcommand{\ee}{\end{equation}}
\theoremstyle{definition}

\renewcommand\appendix{\par
    \setcounter{section}{0}
    \setcounter{lem}{0}
    \setcounter{equation}{0}
   \renewcommand\theequation{A.\arabic{equation}}
   \renewcommand{\thelem}{A.\arabic{lem}}

}

\allowdisplaybreaks

\begin{document}
\title{\bf The existence and asymptotic stability of Plasma-Sheaths to the full Euler-Poisson system}

\author{
Lei Yao\thanks{School of Mathematics and Statistics, Northwestern Polytechnical University, Xi'an 710129, P.R. China. Email: yaolei1056@hotmail.com},\quad Haiyan Yin\thanks{Corresponding
author. School of Mathematical Sciences, Huaqiao
University, Quanzhou 362021, P.R. China. Email:
hyyin@hqu.edu.cn}, \quad Mengmeng Zhu\thanks{School of Mathematics and Center for Nonlinear Studies, Northwest University,
Xi'an 710127, P.R. China. Email: zhumm0907@163.com} }

\date{}
\maketitle
\begin{abstract}
The main concern of this paper is to study large-time behavior of the sheath to the full Euler-Poisson system. As is well known, the monotone stationary solution  under the Bohm criterion can be referred to as the sheath which is formed by interactions of plasma with wall. So far, the existence and asymptotic stability of stationary solutions in one-dimensional half space to the full Euler-Poisson system have been proved in \cite{DYZ2021}. In the present paper, we extend the results in \cite{DYZ2021} to $N$-dimensional ($N$=1,2,3) half space.
By assuming that the velocity of the positive ion satisfies the Bohm criterion at
the far field, we establish the global unique existence and the large time asymptotic stability
of the sheath in some weighted Sobolev spaces by weighted
energy method. Moreover, the time-decay rates are also obtained.
A key different point from \cite{DYZ2021} is to derive some boundary estimates on the derivative of the potential in the $x_1$-direction.
\vspace{4mm}


\end{abstract}

\section{Introduction}

In this paper, we consider the flow of positively charged ions in plasmas over the $N$-dimensional half space $\mathbb{R}_+^N:=\{(x_1,...,x_N)\in \mathbb{R}^N|x_1>0\}$ for $N=1,2,3$. The behavior of the ions is governed by the full Euler-Possion system of the form
\begin{equation}\label{C1}
\left\{
\begin{array}{l}
n_t+\textrm {div}(nu)=0,\\
(mnu)_t+\textrm {div}(mnu\otimes u)+\nabla p=n\nabla\phi,\\
W_t+\textrm{div}(Wu+pu)=nu\cdot\nabla\phi,\\
\Delta \phi=n-e^{-\phi},
\end{array}
\right.
\end{equation}
where unknown functions $n$, $u$, and $\phi$ stand for the density, velocity, and electrostatic potential, respectively.  A positive constant $m$ is the mass of an ion. The function $W$ is given by
\begin{align}\label{C2}
W=\frac{1}{2}mnu^2+\frac{p}{\gamma-1},
\end{align}
where the constant $\gamma>1$ is the ratio of specific heats and the pressure $p$ satisfies the equation of state:
\begin{align}\label{C3}
p=RTn
\end{align}
with the temperature $T$ and the Boltzmann constant $R>0$. Note that $\phi$ is so scaled that it has an opposite sign compared to the usual situation in physics. The fourth equation of \eqref{C1} is the Poisson equation and the electron density $n_e$ is given by $n_e=e^{-\phi}$ under the assumption of the Boltzmann relation. From \cite{DL2015}, we can know that the full Euler-Possion system \eqref{C1} can be formally derived through the macro-micro decomposition from the Vlasov-Possion-Boltzmann system for the ions flow in kinetic.

From \eqref{C2} and \eqref{C3}, we can rewrite \eqref{C1} as follows:
\begin{equation}\label{C4}
\left\{
\begin{array}{l}
n_t+\textrm {div}(nu)=0,\\
mn(u_t+u\cdot\nabla u)+\nabla (RTn)=n\nabla\phi,\\
T_t+u\cdot\nabla T+(\gamma-1)T\textrm {div}u=0,\\
\Delta \phi=n-e^{-\phi}.
\end{array}
\right.
\end{equation}
We put initial condition
\begin{align}\label{C5}
\left\{
\begin{array}{l}
(n,u,T)(0,x)=(n_0,u_0,T_0)(x), \quad \displaystyle\inf_{x\in \mathbb{R}_+^N} n_0(x)>0, \quad \displaystyle\inf_{x\in \mathbb{R}_+^N} T_0(x)>0,\\
\displaystyle\lim_{x_1\to +\infty}(n_0,u_0,T_0)(x_1,x')=(n_\infty, U_\infty, T_\infty)\in \mathbb{R}^{1+N+1}, \quad  x'=(x_2,...,x_N)\in\mathbb{R}^{N-1},
\end{array}
\right.
\end{align}
and boundary condition
\begin{equation}\label{C6}
\phi(t,0,x')=\phi_b\neq 0, \quad \lim_{x_1\to +\infty}\phi(t,x_1,x')=0, \quad x'\in\mathbb{R}^{N-1},
\end{equation}
where $U_\infty=(u_\infty,0,...,0)\in \mathbb{R}^{N}$ and $n_\infty>0$, $u_\infty<0$, $T_\infty>0$ and $\phi_b$ are constants.  To the end, we always assume that
\begin{align}\label{C7}
n_\infty=1,
\end{align}
so that the quasi-neutrality holds true at $x_1=+\infty$ owing to \eqref{C6} and \eqref{C7}.

The main concern of this paper is to study the asymptotic stability of a
plasma boundary layer, called as a sheath to the full Euler-Possion system. The sheath
appears when a material is surrounded by a plasma and the plasma contacts with
its surface. Because the thermal velocities of electrons are much higher than those
of ions, more electrons tend to hit the material compared with ions. This makes
the material negatively charged with respect to the surrounding plasma. Then the
material with a negative potential attracts and accelerates ions toward the surface,
while repelling electrons away from it. Eventually, there appears a non-neutral
potential region near the surface, where a nontrivial equilibrium of the densities is
achieved. This non-neutral region is referred as to the sheath. For more details of the sheath development, we
refer the reader to \cite{CFF1984, KUR1991, KUR1995}. The relevant mathematical study has attracted the attention of many mathematicians after the pioneering work of Langmuir \cite{IL1929,LTIL1929} which reveals the basic features of the plasma sheath transition. Then, Bohm \cite{DB1949} provided the explicit condition and clear interpretation for the formation of sheath, now known as the Bohm criterion. In a review paper \cite{KUR1991}, there are several kinds of the Bohm criterion according to the model. For the full Euler-Poisson equations \eqref{C4} in the present paper, the subsequent condition \eqref{C94} is called the Bohm criterion which indicate that ions must move toward the wall at infinity with a velocity greater than a critical value given particularly as the acoustic velocity for cold ions.

Mathematically, the plasma sheath is described as the stationary solution under the Bohm criterion and it is well-known that there are some mathematical works related to the study of the subject in this paper, we can refer to \cite{A2006, AMR2001, DYZ2021, HS2003, NOS2012, MO2015, MS2011, MS2016, MS2021} and the references therein. Precisely, Ambroso, M${\rm \acute{e}}$hats, and Raviart \cite{AMR2001} gave the existence of the monotone stationary solution by studying a singular perturbation problem for the nonlinear poisson equation. Then Ambroso \cite{A2006} considered a further study to determine the stationary solutions in terms of different levels of an associated energy functional and numerically show which solution is asymptotically stable in large time. Later, Suzuki \cite{MS2011} mathematically proved the above results, in other words, the time asymptotic stability of the monotone stationary solution to the isentropic Euler-Poisson system with the Dirichlet boundary condition was first proved over one-dimensional half space and the stability result requires a condition slightly stronger than the Bohm criterion to hold. Nishibata, Ohnawa, and Suzuki \cite{NOS2012} refined the result in \cite{MS2011} by proving the stability exactly under the Bohm criterion in the spatial dimension up to three and also deal with the degenerate case in which the Bohm criterion is marginally fulfilled.
For a multicomponent plasma containing
electrons and several components of ions, similar results to \cite{NOS2012,MS2011} were obtained
in \cite{MS2016} under the generalized Bohm criterion derived by Riemann in \cite{KUR1995}. Recently, Duan, Yin, and Zhu \cite{DYZ2021} obtained the existence of stationary solutions under the Bohm criterion and further obtained the time asymptotic stability of the monotone stationary solution for the full Euler-Poisson system with the Dirichlet boundary condition.
Moreover, Ohnawa \cite{MO2015} considered the isentropic Euler-Poisson system with the fluid-boundary interaction and gave the existence and asymptotic stability of the monotone stationary solution.  Suzuki and Takayama \cite{MS2021} study the existence and stability of stationary solutions for the isentropic Euler-Poisson system over the domain with the curved boundary. These
results validated mathematically the Bohm criterion and defined the fact that the
sheath corresponds to the stationary solution.

Let us also mention the results on the quasi-neutral limit problem as letting
the Debye length in the Euler-Poisson equations tend to zero.  G$\acute{e}$rard-Varet, Han-Kwan, and Rousset in \cite{Quas13, Quas14} studied the problems over the half space with various boundary conditions. In particular, the result
in \cite{Quas13} clarified the fact that the thickness of the boundary layer is of order of the Debye length.
The time-global solvability and quasi-neutral limit problem were investigated in \cite{Guo2011}
and \cite{quasi1}, respectively. The traveling wave solutions in the quasi-neutral limit problem were established in \cite{quasi2}.

Before closing this section, we give our notation used throughout this paper.
$C$ and $c$ denote some positive constants and they may take different values in different places. $[a](a\in \mathbb{R})$ denotes a maximum integer which does not exceed $a$.  For a nonnegative integer $k\geq 0$, $H^k(\mathbb{R}_+^N)$ denotes the $k$-th order Sobolev space in the $L^2$ sense, equipped with the norm $\Vert\cdot\Vert_k(=\Vert\cdot\Vert_{H^k})$. When $k$ = 0, we note $H^0=L^2$ and $\|\cdot\|:=\|\cdot\|_{L^2}$. $C^l([0,T];H^k(\mathbb{R}_+^N))$ denotes the space of the $l$-times continuously differentiable functions on the interval $[0,T]$ with values in $H^k(\mathbb{R}_+^N)$. Define the following function space
\begin{align*}
\mathscr{X}_i^j([0,T]):=\bigcap\limits_{k=0}^iC^k([0,T];H^{j+i-k}(\mathbb{R}_+^N)),\ \ i,j\in \mathbb{Z},\ \ i,j\geq 0.
\end{align*}
 A norm with an algebraic weight is defined by
\begin{align*}
\Vert f\Vert_{\alpha,\beta,i}&:=\left(\int W_{\alpha,\beta}\sum_{\lvert s\lvert\leq i}(\partial^s f)^2dx\right)^{\frac{1}{2}},\ \ i\in \mathbb{Z},\ \ i\geq 0,\\
W_{\alpha,\beta}&:=(1+\beta x_1)^\alpha,\ \ \alpha,\beta\in \mathbb{R},\ \ \beta>0
\end{align*}
and this norm is equivalent to the norm defined by $\Vert (1+\beta x_1)^{\frac{\alpha}{2}}f\Vert_{H^i}$. For simplicity, we often omit the last subscript $i$ when $i=0$, that is, $\Vert f\Vert_{\alpha,\beta}:=\Vert f\Vert_{\alpha,\beta,0}$.

\subsection{Main results}

Now we introduce the planar stationary solution $(\tilde{n}, \tilde{u}, 0,...,0, \tilde{T},\tilde{\phi})(x_1)$  which is a solution to \eqref{C4} independent of the time variable $t$ and of the tangential coordinates $x'$ in the half space. Therefore, the planar stationary solution $(\tilde{n}, \tilde{u}, 0,...,0, \tilde{T},\tilde{\phi})(x_1)$ satisfies the system
\begin{equation}\label{C8}
\left\{
\begin{array}{l}
(\tilde{n}\tilde{u})_{x_1}=0,\\
m\tilde{n}\tilde{u}\tilde{u}_{x_1}+(R\tilde{T}\tilde{n})_{x_1}=\tilde{n}\tilde{\phi}_{x_1},\\
\tilde{u}\tilde{T}_{x_1}+(\gamma-1)\tilde{T}\tilde{u}_{x_1}=0,\\
\tilde{\phi}_{x_1x_1}=\tilde{n}-e^{-\tilde{\phi}}
\end{array}
\right.
\end{equation}
with conditions \eqref{C5}$-$\eqref{C7}, that is
\begin{equation}\label{C9}
\left\{
\begin{array}{l}
\displaystyle\inf_{x_1\in \mathbb{R}_+} \tilde{n}(x_1)>0, \quad \inf_{x\in \mathbb{R}_+} \tilde{T}(x_1)>0,\\
\displaystyle\lim_{x_1\to +\infty}(\tilde{n}, \tilde{u}, \tilde{T},\tilde{\phi})(x_1)=(1,u_\infty,T_\infty,0),\\
\tilde{\phi}(0)=\phi_b.
\end{array}
\right.
\end{equation}

In the discussion of the existence of the stationary solution, the Sagdeev potential
\begin{align}\label{C10}
V(\phi):=\int_0^\phi\left[f^{-1}(\eta)-e^{-\eta}\right]d\eta, \quad f(n)=\frac{\gamma RT_\infty}{\gamma-1}\left(n^{\gamma-1}-1\right)+\frac{mu_\infty^2}{2}\left(\frac{1}{n^2}-1\right)
\end{align}
plays a crucial role. Here the inverse function $f^{-1}$ in \eqref{C10} is defined by adopting the branch which contains the far-field equilibrium state $(\tilde{n},\tilde{\phi})=(1,0)$. Also, the first and third equations of \eqref{C8} together with the boundary condition \eqref{C9} give that
\begin{align}\label{C10c}
\tilde{T}=T_\infty\tilde{n}^{\gamma-1}.
\end{align}

The unique existence of the monotone stationary solution obtained in \cite{DYZ2021} which deals with one-dimensional problem and we list the main results in the following.

\begin{lem}\label{Clem1}{\rm (see \cite{DYZ2021}).}
Consider the boundary-value problem \eqref{C8}$-$\eqref{C9}.
\item {\rm(i)}\ \ Let $u_\infty$ be a constant satisfying
\begin{align*}
either\enspace u_\infty^2\leq\frac{\gamma RT_\infty}{m}\enspace or \enspace\frac{\gamma RT_\infty+1}{m}\leq u_\infty^2.
\end{align*}
Then the stationary problem \eqref{C8}$-$\eqref{C9} has a unique monotone solution $(\tilde{n}, \tilde{u}, \tilde{T},\tilde{\phi})(x_1)$ verifying
 \begin{align*}
\tilde{n}, \tilde{u}, \tilde{T},\tilde{\phi} \in C({\overline{\mathbb{R}}_+}),\quad \tilde{n}, \tilde{u}, \tilde{T},\tilde{\phi}, \tilde{\phi}_{x_1}\in C^1({\mathbb{R}_+})
\end{align*}
if and only if the boundary data $\phi_b$ satisfies conditions
 \begin{align*}
V(\phi_b)\geq 0,\quad \phi_b\geq f(c_\infty),
\end{align*}
where $c_\infty=\left(\frac{mu_\infty^2}{\gamma RT_\infty}\right)^{\frac{1}{\gamma+1}}$ is the only critical point of $f$.
\item {\rm(ii)}\ \ Let $u_\infty$ be a constant satisfying
\begin{align*}
\frac{\gamma RT_\infty}{m}<u_\infty^2<\frac{\gamma RT_\infty+1}{m}.
\end{align*}
If $\phi_b\neq 0$, then the stationary problem \eqref{C8}$-$\eqref{C9} does not admit any solutions in the function space $C^1(\mathbb{R}_+)$. If $\phi_b=0$, then a constant state $(\tilde{n}, \tilde{u}, \tilde{T},\tilde{\phi})=(1,u_\infty, T_\infty,0)$ is the unique solution.

Moreover, the existing stationary solution enjoys some additional space-decay properties in the following two case:\\
$\bullet$\enspace (Nondegenerate case)\ \ Assume that
\begin{align*}
\frac{\gamma RT_\infty+1}{m}<u_\infty^2,\quad u_\infty<0,
\end{align*}
and $\phi_b\neq f(c_\infty)$ hold true. The stationary solution $(\tilde{n}, \tilde{u}, \tilde{T},\tilde{\phi})(x_1)$ belongs to $C^\infty(\overline{\mathbb{R}}_+)$ and verifies
\begin{align}\label{C11}
\lvert\partial_{x_1}^i(\tilde{n}-1)\rvert+\lvert\partial_{x_1}^i(\tilde{u}-u_\infty)\rvert+\lvert\partial_{x_1}^i(\tilde{T}-T_\infty)\rvert
+\lvert\partial_{x_1}^i\tilde{\phi}\rvert\leq C\lvert\phi_b\rvert e^{-cx_1}
\end{align}
for any $i\geq 0$, where $c$ and $C$ are positive constants.\\
$\bullet$\enspace (Degenerate case)\ \ Assume that
\begin{align*}
\frac{\gamma RT_\infty+1}{m}=u_\infty^2,\quad u_\infty<0,
\end{align*}
and $\phi_b>0$ hold true. Denote constants
\begin{equation*}
\left\{
\begin{array}{l}
c_0=1,\\
c_1=-2\Gamma,\\
c_2=\frac{(\gamma^2+\gamma)RT_\infty+2}{2},\\
c_3=-2\Gamma[(\gamma^2+\gamma)RT_\infty+2]
\end{array}
\right.
\end{equation*}
with
\begin{align}\label{C12}
 \Gamma=\sqrt{\frac{(\gamma^2+\gamma)RT_\infty+2}{12}}.
\end{align}
There are constants $\delta_0>0$ and $C>0$ such that for any $\phi_b\in(0,\delta_0)$,
\begin{align}\label{C13}
\sum_{i=0}^3\Vert\partial_{x_1}^iUG^{i+2}+c_i\Vert_{L^\infty}\leq C\phi_b
\end{align}
with
\begin{align*}
U=-\tilde{\phi},\quad \tilde{n}-1,\quad \log\tilde{n},\quad \frac{\tilde{u}}{u_\infty}-1,\quad \frac{1}{\gamma}\left(\frac{\tilde{T}}{T_\infty}-1\right),
\end{align*}
where $G=G(x_1)$ is a function of the form
\begin{align}\label{C14}
G(x_1)=\Gamma x_1+\phi_b^{-\frac{1}{2}}.\frac{}{}
\end{align}
\end{lem}

From the above lemma, it can be see that the condition that either
\begin{align}\label{C92}
\lvert\phi_b\rvert\ll1,\ \ u_\infty<-\sqrt{\frac{\gamma RT_\infty+1}{m}}
\end{align}
or
\begin{align}\label{C93}
0<\phi_b\ll 1,\ \ u_\infty=-\sqrt{\frac{\gamma RT_\infty+1}{m}}
\end{align}
is sufficient for the unique existence of the nontrivial monotone planar stationary solution.
Also, the condition
\begin{align}\label{C94}
\frac{\gamma RT_\infty+1}{m}\leq u_\infty^2,\quad u_\infty<0,
\end{align}
is called the Bohm criterion.

To study the asymptotic stability of the stationary solution $(\tilde{n}, \tilde{u}, 0,...,0, \tilde{T},\tilde{\phi})(x_1)$, it is convenient to employ unknown functions $v:=\log n$, $\tilde{v}:=\log \tilde{n}$ and perturbations
\begin{align*}
(\varphi,\psi,\zeta,\sigma)(t,x_1,x'):=(v,u,T,\phi)(t,x_1,x')-(\tilde{v},\tilde{U},\tilde{T},\tilde{\phi})(x_1),
\end{align*}
where $\tilde{U}:=(\tilde{u},0,...,0)$. Then we can reformulate the problem \eqref{C4}$-$\eqref{C7} in the framework of perturbations as follows:
\begin{equation}\label{C15}
\left\{
\begin{array}{l}
\varphi_t+u\cdot \nabla\varphi+{\rm div}\psi=-\psi_1\tilde{v}_{x_1},\\
m(\psi_t+u\cdot\nabla\psi)+RT\nabla\varphi+R\nabla\zeta=-m\psi_1\tilde{U}_{x_1}-R\zeta\nabla\tilde{v}+\nabla\sigma,\\
\zeta_t+u\cdot\nabla\zeta+(\gamma-1)T{\rm div}\psi=-\psi_1\tilde{T}_{x_1}-(\gamma-1)\zeta\tilde{u}_{x_1},\\
\Delta\sigma=e^{\varphi+\tilde{v}}-e^{\tilde{v}}-e^{-(\sigma+\tilde{\phi})}+e^{-\tilde{\phi}}.
\end{array}
\right.
\end{equation}
The initial and boundary conditions for $(\varphi,\psi,\zeta,\sigma)(t,x)$ follow from $\eqref{C5}-\eqref{C7}$ and \eqref{C9} as
\begin{align}\label{C16}
(\varphi,\psi,\zeta)(0,x)=(\varphi_0,\psi_0,\zeta_0)(x):=(\log n_0-\log\tilde{n},u_0-\tilde{U},T_0-\tilde{T}),
\end{align}
\begin{align}\label{C17}
\sigma(t,0,x')=0,\quad  x'\in \mathbb{R}^{N-1}.
\end{align}

If the perturbation $(\varphi,\psi,\zeta)(t,x)$ and $\lvert\phi_b\rvert$ are sufficiently small, all characteristics in the $x_1$-direction of hyperbolic system $\eqref{C15}_1$, $\eqref{C15}_2$ and $\eqref{C15}_3$ are negative owing to \eqref{C94}:
\begin{equation}\label{C18}
\left\{
\begin{array}{l}
\lambda_1[u_1,T]:=\frac{(m+1)u_1-\sqrt{(m-1)^2u_1^2+4\gamma RT}}{2}<0,\\
\lambda_2[u_1,T]:=u_1<0,\\
\lambda_3[u_1,T]:=\frac{(m+1)u_1+\sqrt{(m-1)^2u_1^2+4\gamma RT}}{2}<0,\\
\lambda_i[u_1,T]:=mu_1<0,\quad {\rm for}\ i=4,...,N+1.
\end{array}
\right.
\end{equation}
Hence, no boundary conditions for the hyperbolic system $\eqref{C15}_1$, $\eqref{C15}_2$ and $\eqref{C15}_3$ are necessary for the well-posedness of the initial boundary value problem \eqref{C15}$-$\eqref{C17}.

The asymptotic stability of the stationary solution $(\tilde{n}, \tilde{u}, 0,...,0, \tilde{T},\tilde{\phi})(x_1)$ is stated in the following theorems.
\begin{thm}\label{CThm1}{\rm(nondegenerate case)}.
Assume that the condition \eqref{C92} holds and let $s=[N/2]+2$, $N=1,2,3$.
\item {\rm(i)}\ \ The initial condition is supposed to satisfy
\begin{align*}
 (e^{\lambda x_1/2}\varphi_0,e^{\lambda x_1/2}\psi_0,e^{\lambda x_1/2}\zeta_0)\in(H^s(\mathbb{R}_+^N))^{N+2}
\end{align*}
for some positive constant $\lambda$, then there exists a positive constant $\delta$ such that if
\begin{align*}
\beta\in (0,\lambda] \ \ {\rm and} \ \ \beta+(\lvert\phi_b\rvert+\Vert(e^{\lambda x_1/2}\varphi_0,e^{\lambda x_1/2}\psi_0,e^{\lambda x_1/2}\zeta_0)\Vert_{H^s})\big{/}\beta\leq \delta,
\end{align*}
the initial boundary value problem \eqref{C15}$-$\eqref{C17} has a unique global solution $(\varphi,\psi,\zeta,\sigma)(t,x)$ satisfying
\begin{align*}
 (e^{\lambda x_1/2}\varphi,e^{\lambda x_1/2}\psi,e^{\lambda x_1/2}\zeta,e^{\lambda x_1/2}\sigma)\in (\mathscr{X}_s^0(\mathbb{R}_+))^{N+2}\times\mathscr{X}_s^2(\mathbb{R}_+).
\end{align*}
Moreover, the solution $(\varphi,\psi,\zeta,\sigma)(t,x)$ verifies the decay estimate
\begin{align*}
\Vert(e^{\lambda x_1/2}\varphi,e^{\lambda x_1/2}\psi,e^{\lambda x_1/2}\zeta)(t)\Vert_{H^s}^2+\Vert e^{\lambda x_1/2}\sigma(t)\Vert_{H^{s+2}}^2\leq C\Vert(e^{\lambda x_1/2}\varphi_0,e^{\lambda x_1/2}\psi_0,e^{\lambda x_1/2}\zeta_0)\Vert_{H^s}^2e^{-\mu t},
\end{align*}
where $C$ and $\mu$ are positive constants independent of $t$.
\item {\rm(ii)}\ \ Assume $\lambda\geq 2$ holds. The initial condition is supposed to satisfy
\begin{align*}
 ((1+\beta x_1)^{\lambda /2}\varphi_0,(1+\beta x_1)^{\lambda /2}\psi_0,(1+\beta x_1)^{\lambda /2}\zeta_0)\in(H^s(\mathbb{R}_+^N))^{N+2}
\end{align*}
for $\beta>0$, then there exists a positive constant $\delta$ such that if
\begin{align*}
\beta+(\lvert\phi_b\rvert+\Vert((1+\beta x_1)^{\lambda /2}\varphi_0,(1+\beta x_1)^{\lambda /2}\psi_0,(1+\beta x_1)^{\lambda /2}\zeta_0)\Vert_{H^s})\big{/}\beta\leq \delta,
\end{align*}
the initial boundary value problem \eqref{C15}$-$\eqref{C17} has a unique global solution $(\varphi,\psi,\zeta,\sigma)(t,x)$ satisfying
\begin{align*}
 ((1+\beta x_1)^{\varepsilon /2}\varphi,(1+\beta x_1)^{\varepsilon /2}\psi,(1+\beta x_1)^{\varepsilon /2}\zeta,(1+\beta x_1)^{\varepsilon /2}\sigma)\in (\mathscr{X}_s^0(\mathbb{R}_+))^{N+2}\times\mathscr{X}_s^2(\mathbb{R}_+),
\end{align*}
where $\varepsilon\in (0,\lambda]$.
Moreover, the solution $(\varphi,\psi,\zeta,\sigma)(t,x)$ verifies the decay estimate
\begin{align*}
\Vert((1+\beta x_1)^{\varepsilon /2}\varphi,&(1+\beta x_1)^{\varepsilon /2}\psi,(1+\beta x_1)^{\varepsilon /2}\zeta)(t)\Vert_{H^s}^2+\Vert (1+\beta x_1)^{\varepsilon /2}\sigma(t)\Vert_{H^{s+2}}^2\nonumber\\
&\leq C\Vert((1+\beta x_1)^{\lambda /2}\varphi_0,(1+\beta x_1)^{\lambda /2}\psi_0,(1+\beta x_1)^{\lambda /2}\zeta_0)\Vert_{H^s}^2(1+\beta t)^{-(\lambda-\varepsilon)},
\end{align*}
where $C$ is a positive constant independent of $t$.
\end{thm}

\begin{thm}\label{CThm2}{\rm(degenerate case)}. Assume that the condition \eqref{C93} holds and let $s=[N/2]+2$, $N=1,2,3$.
Let $4<\lambda_0<5.5693\cdots$ be the unique real solution to the equation
\begin{align}\label{C95}
 \lambda_0(\lambda_0-1)(\lambda_0-2)-12\left(\frac{2}{1+\gamma}\lambda_0+2\right)=0,
\end{align}
where $5.5693\cdots$ is the unique real solution to the equation
\begin{align}\label{C96}
  \lambda_0(\lambda_0-1)(\lambda_0-2)-12(\lambda_0+2)=0
\end{align}
and $\lambda\in [4,\lambda_0)$ is satisfied. For $\Gamma=\sqrt{\frac{(\gamma^2+\gamma)RT_\infty+2}{12}}$ and $\theta\in (0,1]$, there exists a positive constant $\delta$ such that if
$\phi_b \in (0,\delta]$, $\beta/(\Gamma \phi_b^{1/2})\in[\theta,1]$,
\begin{align*}
((1+\beta x_1)^{\lambda /2}\varphi_0,(1+\beta x_1)^{\lambda /2}\psi_0,(1+\beta x_1)^{\lambda /2}\zeta_0)\in(H^s(\mathbb{R}_+^N))^{N+2}
\end{align*}
and
\begin{align*}
\Vert((1+\beta x_1)^{\lambda /2}\varphi_0,(1+\beta x_1)^{\lambda /2}\psi_0,(1+\beta x_1)^{\lambda /2}\zeta_0)\Vert_{H^s})/\beta^3\leq\delta
\end{align*}
are satisfied, the initial boundary value problem \eqref{C15}$-$\eqref{C17} has a unique global solution $(\varphi,\psi,\zeta,\sigma)(t,x)$ satisfying
\begin{align*}
 ((1+\beta x_1)^{\varepsilon /2}\varphi,(1+\beta x_1)^{\varepsilon /2}\psi,(1+\beta x_1)^{\varepsilon /2}\zeta,(1+\beta x_1)^{\varepsilon /2}\sigma)\in (\mathscr{X}_s^0(\mathbb{R}_+))^{N+2}\times\mathscr{X}_s^2(\mathbb{R}_+),
\end{align*}
where $\varepsilon\in (0,\lambda]$. Moreover, the solution $(\varphi,\psi,\zeta,\sigma)(t,x)$ verifies the decay estimate
\begin{align*}
\Vert((1+\beta x_1)^{\varepsilon /2}\varphi,&(1+\beta x_1)^{\varepsilon /2}\psi,(1+\beta x_1)^{\varepsilon /2}\zeta)(t)\Vert_{H^s}^2+\Vert (1+\beta x_1)^{\varepsilon /2}\sigma(t)\Vert_{H^{s+2}}^2\nonumber\\
&\leq C\Vert((1+\beta x_1)^{\lambda /2}\varphi_0,(1+\beta x_1)^{\lambda /2}\psi_0,(1+\beta x_1)^{\lambda /2}\zeta_0)\Vert_{H^s}^2(1+\beta t)^{-(\lambda-\varepsilon)/3},
\end{align*}
where $C$ is a positive constant independent of $t$.
\end{thm}

The main concern of this paper is to study large-time behavior of the sheath to the full Euler-Poisson system.  So far, the existence and asymptotic stability of stationary solutions in one-dimensional half space to the full Euler-Poisson system have been proved in \cite{DYZ2021}. In the present paper, we extend the results in \cite{DYZ2021} to $N$-dimensional ($N$=1,2,3) half space.
By assuming that the velocity of the positive ion satisfies the Bohm criterion at
the far field, we establish the global unique existence and the large time asymptotic stability
of the sheath in some weighted Sobolev spaces by weighted
energy method. Moreover, the time-decay rates are also obtained. In fact, in comparison with the related one-dimensional results
in \cite{DYZ2021}, we need to make additional efforts to consider the effect of the spatial dimension $N=2$ or $N=3$ in the proof. Here, let us discuss more details of the difficulty and the strategies to
resolve as follows:

$\bullet$ The global existence for nonlinear hyperbolic systems usually can be obtained in the presence of dissipation terms or decay properties, but the Euler-Poisson system under consideration is not the case. Thus, the stability analysis for the Euler-Poisson system can not be handled by the standard methods and we need to borrow some ideas from \cite{NNY2007} which also studies an outflow problem for the Navier-Stokes equation. The main borrowing approach is that
we use two types of the weight functions as $(1+\beta x_1)^\lambda$ and $e^{\lambda x_1}$ to do our energy estimates.

$\bullet$ In the framework of the energy method, we first focus on the estimate up to the first order derivatives. It is easy to see that Lemma \ref{Clem3} is sufficient for the one-dimensional problem \cite{DYZ2021} since $\lvert {\rm div}\psi\rvert =\lvert \nabla\psi\rvert$. However, for the spatial dimension $N=2$ or $N=3$, Lemma \ref{Clem3} is insufficient to close the estimates due to the term $\int_0^t(1+\beta\tau)^\xi\beta\Vert\nabla\psi\Vert_{\varepsilon-1,\beta}^2d\tau$ appears on the right-hand side of  \eqref{C102}. Therefore, we then proceed to prove Lemma \ref{Clem4} which gives the estimates of $\int_0^t(1+\beta\tau)^\xi\beta\Vert\nabla\psi\Vert_{\varepsilon-1,\beta}^2d\tau$.

$\bullet$ For
the present problem (the spatial dimension $N=2$ or 3), the following boundary term $$\int_{\mathbb{R}^{N-1}}\sigma_{x_1}^2(t,0,x')dx'$$ (see \eqref{C60} for more details) appears on the right-hand side of \eqref{C103} in the derivation of Lemma \ref{Clem4}. The way of getting around the difficulty explained above is to make full use of the key observation \eqref{C65} to capture the positivity of the above boundary term, which is motivated by the works of \cite{NOS2012,MS2011}. More precisely, we first apply $e^{-v}\partial_t$ to $\eqref{C15}_4$ and multiply the resulting equation by $\sigma$ to obtain the key term $\tilde{u}\sigma_{x_1}(e^\varphi-1)$. Then by using $\eqref{C15}_4$ again, we can rewrite $$\tilde{u}\sigma_{x_1}(e^\varphi-1)=\tilde{u}\sigma_{x_1}\left\{e^{-\tilde{v}}\Delta\sigma+e^{-\tilde{v}-\tilde{\phi}}(e^{-\sigma}-1)\right\}.$$
The positivity of the boundary term can be derived by multiplying the above equality by a weighted function  and integrating the resultant over the $\mathbb{R}_+^N$, where we mainly utilize integration by parts and the boundary condition \eqref{C17}. For more details, we can see Lemma \ref{Clem5}.

$\bullet$ Moreover, it is apparent that we need also give a \emph{priori} estimates for higher orders. To this end, we simplify the estimates of higher orders by analyzing the structure of the equation \eqref{C15} and using Lemma \ref{Clem14}. Meanwhile, we are free to apply the same derivation as \eqref{C15} to \eqref{C77} since the boundary condition \eqref{C137} holds. Consequently, we only sketch the proof in Lemma \ref{Clem6}$-$Lemma \ref{Clem8}.

The rest of the paper is arranged as follows. Section 2 contains the energy estimates that will also give the decay estimate of the solution for the degenerate case. Such an argument would be more lengthy than the nondegenerate case and we make full use of the time-space weighted energy estimates to complete the proof of Theorem \ref{CThm2}. In section 3, we give the energy estimates for the nondegenerate case and complete the proof of Theorem \ref{CThm1}. In Appendix A, we record various analytic tools that are useful throughout the paper.

\section{Energy estimates for the degenerate case}

In this section, we study the stability of the planar stationary solution to \eqref{C4} for the degenerate case \eqref{C93}, where the Bohm criterion is marginally fulfilled. In what follows, we focus our attention on the condition that $\phi_b>0$ is suitably small to ensure the existence of a nontrivial monotone planar stationary solution to \eqref{C8}$-$\eqref{C9}.

Next, we summarize the local existence of the problem  \eqref{C15}$-$\eqref{C17} in the following lemma which is proved by a similar method as in \cite{MS2011} and \cite{NOS2012}. Here, the details are omitted for simplicity of presentation. Then the global existence of solution to the problem \eqref{C15}$-$\eqref{C17} will be obtained by combining the local existence result with a \emph{priori} estimates by the standard continuity argument.
\begin{lem}\label{Clem2}\rm{(Local existence).}
Let $0<\phi_b\ll 1$ and $u_\infty\leq-\sqrt{\frac{\gamma RT_\infty+1}{m}}$. Denote $W_\alpha=(1+\alpha x_1)^\lambda$ or $e^{\alpha x_1}$ with a constant $\alpha>0$ suitably small and a certain constant $\lambda>0$. Suppose the initial condition satisfies $(W_\alpha^{1/2}\varphi_0,W_\alpha^{1/2}\psi_0,W_\alpha^{1/2}\zeta_0)\in (H^s(\mathbb{R}_+^N))^{N+2}$ with $\inf_{\mathbb{R}_+^N}(\zeta_0+\tilde{T})>0$ and
\begin{align*}
 \lambda_3[(\psi_0+\tilde{u})(x),(\zeta_0+\tilde{T})(x)]<0,\ \ \forall x\in \mathbb{R}_+^N,
\end{align*}
where $\lambda_3[\cdot,\cdot]$ is defined in the third equation of \eqref{C18}. Let $0<\beta\leq\alpha$. Then, there exists a constant $T>0$ such that the initial boundary value problem  \eqref{C15}$-$\eqref{C17} admits a unique solution $(\varphi,\psi,\zeta,\sigma)(t,x)$ satisfying
\begin{align*}
(W_\beta^{1/2}\varphi,W_\beta^{1/2}\psi,W_\beta^{1/2}\zeta,W_\beta^{1/2}\sigma)\in (\mathscr{X}_s^0([0,T]))^{N+2}\times\mathscr{X}_s^2([0,T])
\end{align*}
as well as the condition $\eqref{C18}$ on $\{x_1=0\}$.
\end{lem}

\subsection{A \emph{priori} estimates for the degenerate case}

The section is devoted to show a \emph{priori} estimates of Proposition \ref{CProp1}. To this end, we define the following notation for convenience:
\begin{align*}
\vartheta(t,x):=(\varphi,\psi,&\zeta)(t,x),\quad \vartheta_0(x):=(\varphi_0,\psi_0,\zeta_0)(x),\\
N_{\alpha,\beta}(T):&=\sup_{0\leq t\leq T}\Vert\vartheta(t)\Vert_{\alpha,\beta,[N/2]+2}.
\end{align*}

\begin{prop}\label{CProp1}
Assume the same conditions on $N,m,T_\infty,u_\infty, \lambda_0$, and $\lambda$ hold as in Theorem \ref{CThm2}. Let $(\vartheta,\sigma)(t,x)$ be a solution to \eqref{C15}$-$\eqref{C17} over a time interval $[0,T]$ for  $T>0$. For $\Gamma=\sqrt{\frac{(\gamma^2+\gamma)RT_\infty+2}{12}}$ and $\theta\in (0,1]$, there exist positive constants $\delta$ and $C$ independent of $T$ such that if all the following conditions
\begin{align}\label{C97}
\phi_b \in (0,\delta],
\end{align}
\begin{align}\label{C98}
\beta/(\Gamma \phi_b^{1/2})\in[\theta,1],
\end{align}

\begin{align}\label{C100}
((1+\beta x_1)^{\lambda /2}\vartheta,(1+\beta x_1)^{\lambda /2}\sigma)\in (\mathscr{X}_s^0([0,T]))^{N+2}\times\mathscr{X}_s^2([0,T]),
\end{align}
and
\begin{align}\label{C101}
N_{\lambda,\beta}(T)/\beta^3\leq\delta
\end{align}
are satisfied, then it holds for any $\varepsilon\in (0,\lambda]$ and $t\in [0,T]$ that
\begin{align}\label{C102-1}
\Vert\vartheta(t)\Vert_{\varepsilon,\beta,s}^2+\Vert\sigma(t)\Vert_{\varepsilon,\beta,s+2}^2\leq C\Vert\vartheta_0\Vert_{\lambda,\beta,s}^2(1+\beta t)^{-(\lambda-\varepsilon)/3}.
\end{align}
\end{prop}

The proof of Proposition \ref{CProp1} will be proved by following Lemma \ref{Clem3}$-$Lemma \ref{Clem8} at the end of this section.
\begin{lem}\label{Clem3}
 Under the same assumption as in Proposition \ref{CProp1}, it holds for any $\xi\geq0$ and $t\in [0,T]$ that
\begin{align}\label{C102}
(1+\beta t)^\xi&\Vert(\vartheta,\nabla\varphi,{\rm div} \psi,\nabla\zeta)(t)\Vert_{\varepsilon,\beta}^2\nonumber\\
&+\int_0^t(1+\beta\tau)^\xi\left\{\beta^3\Vert\vartheta(\tau)\Vert_{\varepsilon-3,\beta}^2+\beta\Vert(\nabla\varphi,{\rm div}\psi,\nabla\zeta,\nabla\sigma)(\tau)\Vert_{\varepsilon-1,\beta}^2\right\}d\tau\nonumber\\
\leq C\Vert\vartheta_0\Vert_{\varepsilon,\beta,1}^2&+C\xi\int_0^t(1+\beta\tau)^{\xi-1}\beta\Vert\vartheta(\tau)\Vert_{\varepsilon,\beta,1}^2d\tau+C\delta\int_0^t
(1+\beta\tau)^\xi\beta\Vert\nabla\psi(\tau)\Vert_{\varepsilon-1,\beta}^2d\tau,
\end{align}
where positive constants $\delta$ and $C$ are independent of $T$.
\end{lem}
\begin{proof}
Multiplying $RT\eqref{C15}_1, \eqref{C15}_2$, and $\frac{R}{(\gamma-1)T}\eqref{C15}_3$ by $\tilde{n}\varphi$, $\tilde{n}\psi$, and $\tilde{n}\zeta$ respectively, using $\tilde{v}_{x_1}=\frac{\tilde{n}_{x_1}}{\tilde{n}}$, we get
\begin{align}\label{C19}
(\mathcal{E}_0)_t+{\rm div}\mathcal{H}_0+\mathcal{D}_0+\tilde{n}{\rm div}\psi\sigma=\mathcal{R}_0,
\end{align}
where we have denoted
\begin{align}\label{C20}
\mathcal{E}_0=\frac{\tilde{n}}{2}RT\varphi^2+\frac{\tilde{n}}{2}m\psi^2+\frac{\tilde{n}R}{2(\gamma-1)T}\zeta^2,
\end{align}
\begin{align}\label{C21}
\mathcal{H}_0=\frac{\tilde{n}}{2}RTu\varphi^2+\tilde{n}RT\varphi\psi+\frac{\tilde{n}}{2}mu\psi^2+R\tilde{n}\psi\zeta+\frac{R\tilde{n}u}{2(\gamma-1)T}\zeta^2
-\tilde{n}\sigma\psi,
\end{align}
\begin{align}
\mathcal{D}_0=&-\left(\frac{RTu_1}{2}\tilde{n}_{x_1}+\frac{R\tilde{n}u_1}{2}\tilde{T}_{x_1}+\frac{R\tilde{n}T}{2}\tilde{u}_{x_1}\right)\varphi^2-\tilde{n}R\tilde{T}_{x_1}
\varphi\psi_1-(\frac{m\tilde{n}}{2}\tilde{u}_{x_1}+\frac{mu_1}{2}\tilde{n}_{x_1})\psi^2+m\tilde{n}\tilde{u}_{x_1}\psi_1^2\nonumber\\
&+\frac{R\tilde{n}}{(\gamma-1)T}\tilde{T}_{x_1}\psi_1\zeta+\tilde{n}_{x_1}\psi_1\sigma+\left(\frac{R\tilde{n}\tilde{u}_{x_1}}{T}-\frac{Ru_1\tilde{n}_{x_1}
+R\tilde{n}\tilde{u}_{x_1}}{2(\gamma-1)T}+\frac{R\tilde{n}u_1\tilde{T}_{x_1}}{2(\gamma-1)T^2}\right)\zeta^2\label{C22}
\end{align}
and
\begin{align}\label{C23}
\mathcal{R}_0=&\left(\frac{R\tilde{n}\zeta_t}{2}+\frac{R\tilde{n}T{\rm div\psi}}{2}+\frac{R\tilde{n}u\cdot\nabla\zeta}{2}\right)\varphi^2+R\tilde{n}\varphi
 \psi\cdot\nabla\zeta\nonumber\\
&+\frac{m\tilde{n}{\rm div}\psi}{2}\psi^2+\left(\frac{R\tilde{n}{\rm div}\psi}{2(\gamma-1)T}-\frac{R\tilde{n}u\cdot\nabla\zeta}{2(\gamma-1)T^2}-
\frac{R\tilde{n}\zeta_t}{2(\gamma-1)T^2}\right)\zeta^2.
\end{align}

Now, let's focus on the strategy of handing the term $\tilde{n}{\rm div}\psi\sigma$.  Firstly, we multiply $\nabla\left(RT\eqref{C15}_1\right)$, ${\rm div}\eqref{C15}_2$, and $\nabla\left(\frac{R}{(\gamma-1)T}\eqref{C15}_3\right)$ by $\tilde{n}\nabla\varphi$, $\tilde{n}{\rm div}\psi$, and $\tilde{n}\nabla\zeta$ respectively, using $\tilde{v}_{x_1}=\frac{\tilde{n}_{x_1}}{\tilde{n}}$ again, one has
\begin{align}\label{C24}
(\mathcal{E}_1)_t+{\rm div}\mathcal{H}_1-\tilde{n}{\rm div}\psi\Delta\sigma=\mathcal{R}_1
\end{align}
where we also have denoted
\begin{align}\label{C25}
\mathcal{E}_1=\frac{\tilde{n}}{2}RT(\nabla\varphi)^2+\frac{\tilde{n}}{2}m({\rm div}\psi)^2+\frac{\tilde{n}R}{2(\gamma-1)T}\lvert\nabla\zeta\rvert^2,
\end{align}
\begin{align}\label{C26}
\mathcal{H}_1=\frac{\tilde{n}}{2}RTu\lvert\nabla\varphi\rvert^2+\tilde{n}RT\nabla\varphi{\rm div}\psi+\frac{\tilde{n}}{2}mu({\rm div}\psi)^2+R\tilde{n}{\rm div}\psi\nabla\zeta+\frac{R\tilde{n}u}{2(\gamma-1)T}\lvert\nabla\zeta\rvert^2,
\end{align}
and
\begin{align}\label{C27}
\mathcal{R}_1=&\left[\frac{\tilde{n}R}{2}\zeta_t+\frac{RTu_1}{2}\tilde{n}_{x_1}+\frac{\tilde{n}R}{2}{\rm div}(Tu)\right]\lvert\nabla\varphi\rvert^2
+\left[\frac{mu_1}{2}\tilde{n}_{x_1}+\frac{m\tilde{n}}{2}{\rm div}u\right]({\rm div}\psi)^2\nonumber\\
&+\left[\frac{R\tilde{n}}{2(\gamma-1)T}{\rm div}u+\frac{Ru_1}{2(\gamma-1)T}\tilde{n}_{x_1}-\frac{R\tilde{n}u}{2(\gamma-1)T^2}\nabla T-
\frac{R\tilde{n}}{2(\gamma-1)T^2}\zeta_t-\frac{R\tilde{n}}{T}\tilde{u}_{x_1}\right]\lvert\nabla\zeta\rvert^2\nonumber\\
&-\tilde{n}\tilde{u}_{x_1}\left[RT(\partial_{x_1}\varphi)^2+2m\partial_{x_1}\psi_1{\rm div}\psi+\frac{R}{(\gamma-1)T}(\partial_{x_1}\zeta)^2\right]+RT\tilde{n}_{x_1}(\partial_{x_1}\varphi
{\rm div}\psi-\nabla\varphi\cdot\nabla\psi_1)\nonumber\\
&-\tilde{n}\tilde{T}_{x_1}\left[Ru\partial_{x_1}\varphi\cdot\nabla\varphi+\frac{R}{(\gamma-1)T}\nabla\psi_1\cdot\nabla\zeta-\frac{Ru}{(\gamma-1)T^2}
\partial_{x_1}\zeta\cdot\nabla\zeta\right]\nonumber\\
&-R\tilde{n}T\sum_{i,j=1}^3\partial_{x_i}\varphi\partial_{x_j}\varphi\partial_{x_i}\psi_j-R\tilde{n}(u\cdot\nabla\varphi)(\nabla\varphi\cdot\nabla\zeta)
-m\tilde{n}{\rm div}\psi\sum_{i,j=1}^3\partial_{x_i}\psi_j\partial_{x_j}\psi_i\nonumber\\
&-\frac{R\tilde{n}}{(\gamma-1)T}\nabla\zeta\cdot\nabla\psi\cdot\nabla\zeta+\frac{R\tilde{n}}{(\gamma-1)T^2}u\cdot\nabla\zeta\lvert\nabla\zeta\rvert^2-R\tilde{n}\nabla T({\rm div}\psi+\varphi_t)\nabla\varphi\nonumber\\
&+\frac{R\tilde{n}}{(\gamma-1)T^2}\nabla T\cdot\nabla\zeta\zeta_t-R\tilde{n}_{x_1}\psi_1\nabla\varphi\cdot
\nabla T+\frac{R\tilde{n}\tilde{T}_{x_1}}{(\gamma-1)T^2}\psi_1\nabla \zeta\cdot\nabla T+\frac{R\tilde{n}\tilde{u}_{x_1}}{T^2}\zeta\nabla \zeta\cdot\nabla T\nonumber\\
&-R\tilde{n}\tilde{v}_{x_1x_1}(T\psi_1\partial_{x_1}\varphi+\zeta{\rm div}\psi)-\tilde{n}\tilde{u}_{x_1x_1}\left(m\psi_1{\rm div}\psi+\frac{R\tilde{n}}{T}\zeta
\partial_{x_1}\zeta\right)-\frac{R\tilde{n}}{(\gamma-1)T}\tilde{T}_{x_1x_1}\psi_1\partial_{x_1}\zeta.
\end{align}

Secondly, we can using Taylor expansion to rewrite $\eqref{C15}_4$, namely,
\begin{align}\label{C28}
\Delta\sigma=\tilde{n}\left(\varphi+\frac{1}{2}e^{\theta_1\varphi}\varphi^2\right)-e^{-\tilde{\phi}}\left(-\sigma+\frac{1}{2}e^{-\theta_2\sigma}\sigma^2\right),\quad \theta_1,\theta_2\in(0,1),
\end{align}
then using $\eqref{C15}_1$, one has
\begin{align}\label{C29}
\tilde{n}(\sigma e^{-\tilde{\phi}}-\Delta\sigma){\rm div}\psi=&\left(\frac{1}{2}\tilde{n}^2\varphi^2\right)_t+{\rm div}\left(\frac{\tilde{n}^2u\varphi^2}{2}\right)
-\frac{1}{2}\tilde{n}^2{\rm div}\psi\varphi^2-\frac{1}{2}\tilde{n}^2\tilde{u}_{x_1}\varphi^2-\tilde{n}u_1\tilde{n}_{x_1}\varphi^2\nonumber\\
&+\tilde{n}^2\tilde{v}_{x_1}\varphi\psi_1-\frac{\tilde{n}}{2}\left(\tilde{n}e^{\theta_1\varphi}\varphi^2-e^{-(\theta_2\sigma+\tilde{\phi})}\sigma^2\right){\rm div}\psi.
\end{align}

Finally, we multiply \eqref{C19} by $e^{-\tilde{\phi}}$ and add the result to \eqref{C24} with the help of \eqref{C29} to get
\begin{align}\label{C30}
&\left(e^{-\tilde{\phi}}\mathcal{E}_0+\mathcal{E}_1+\frac{1}{2}\tilde{n}^2\varphi^2\right)_t+{\rm div}\left(e^{-\tilde{\phi}}\mathcal{H}_0+\mathcal{H}_1+\frac{\tilde{n}^2u\varphi^2}{2}\right)+e^{-\tilde{\phi}}\tilde{\phi}_{x_1}(\mathcal{H}_0)_1
+e^{-\tilde{\phi}}\mathcal{D}_0\nonumber\\
&-\left(\frac{1}{2}\tilde{n}^2\tilde{u}_{x_1}+\tilde{n}u_1\tilde{n}_{x_1}\right)\varphi^2+\tilde{n}^2\tilde{v}_{x_1}\varphi\psi_1\nonumber\\
=&e^{-\tilde{\phi}}\mathcal{R}_0+\mathcal{R}_1+\frac{1}{2}\tilde{n}^2{\rm div}\psi\varphi^2+\frac{\tilde{n}}{2}\left(\tilde{n}e^{\theta_1\varphi}\varphi^2
-e^{-(\theta_2\sigma+\tilde{\phi})}\sigma^2\right){\rm div}\psi
:=\mathcal{N}_1,
\end{align}
where $(\mathcal{H}_0)_1$ denotes the first component of $\mathcal{H}_0$.

Multiply \eqref{C30} by $W_{\varepsilon,\beta}=(1+\beta x_1)^\varepsilon$ and integrating the resulting equation over $\mathbb{R}_+^N$ to obtain
\begin{align}\label{C31}
\frac{d}{dt}&\int_{\mathbb{R}_+^N}W_{\varepsilon,\beta}\left(e^{-\tilde{\phi}}\mathcal{E}_0+\mathcal{E}_1+\frac{1}{2}\tilde{n}^2\varphi^2\right)dx\nonumber\\
&+\underbrace{\int_{\mathbb{R}_+^N}\varepsilon\beta W_{\varepsilon-1,\beta}\left(-e^{-\tilde{\phi}}(\mathcal{H}_0)_1-\frac{1}{2}\tilde{n}^2u_1\varphi^2\right)dx}_{I_1}
+\underbrace{\int_{x_1=0}\left(-e^{-\tilde{\phi}}(\mathcal{H}_0)_1-\frac{1}{2}\tilde{n}^2u_1\varphi^2\right)dx'}_{I_2}\nonumber\\
&+\underbrace{\int_{\mathbb{R}_+^N}\varepsilon\beta W_{\varepsilon-1,\beta}(-\mathcal{H}_1)_1dx}_{I_3}+\underbrace{\int_{x_1=0}(-\mathcal{H}_1)_1dx'}_{I_4}\nonumber\\
&+\underbrace{\int_{\mathbb{R}_+^N}W_{\varepsilon,\beta}\left[e^{-\tilde{\phi}}\tilde{\phi}_{x_1}(\mathcal{H}_0)_1
+e^{-\tilde{\phi}}\mathcal{D}_0-\left(\frac{1}{2}\tilde{n}^2\tilde{u}_{x_1}
+\tilde{n}u_1\tilde{n}_{x_1}\right)\varphi^2+\tilde{n}^2\tilde{v}_{x_1}\varphi\psi_1\right]dx}_{I_5}\nonumber\\
=&\int_{\mathbb{R}_+^N}W_{\varepsilon,\beta}\mathcal{N}_1dx,
\end{align}
where $(\mathcal{H}_1)_1$ also denotes the first component of $\mathcal{H}_1$.

Firstly, we estimate three terms $I_2$, $I_3$, and $I_4$ on the left-hand side of $\eqref{C31}$. Due to the equality $u_1=\psi_1+(\tilde{u}-u_\infty)+u_\infty$, $T=\zeta+(\tilde{T}-T_\infty)+T_\infty$ and the fact of \eqref{C13} and Lemma \ref{Clem12}, then we can see, under the conditions \eqref{C93} and \eqref{C97}$-$\eqref{C101}, that
\begin{align}\label{C32}
\frac{(-\mathcal{H}_1)_1}{\tilde{n}}&=\frac{-RTu_1}{2}\lvert\nabla\varphi\rvert^2-RT{\rm div}\psi\partial_{x_1}\varphi+\frac{-mu_1}{2}({\rm div}\psi)^2-R{\rm div}\psi\partial_{x_1}\zeta+\frac{-Ru_1}{2(\gamma-1)T}\lvert\nabla\zeta\rvert^2\nonumber\\
\geq &\frac{-RT_\infty u_\infty}{2}\lvert\nabla\varphi\rvert^2-RT_\infty{\rm div}\psi\partial_{x_1}\varphi+\frac{-mu_\infty}{2}({\rm div}\psi)^2-R{\rm div}\psi\partial_{x_1}\zeta+\frac{-Ru_\infty}{2(\gamma-1)T_\infty}\lvert\nabla\zeta\rvert^2\nonumber\\
&-C(N_{\lambda,\beta(T)}+\phi_b)(\lvert\nabla\varphi\rvert^2+\lvert{\rm div}\psi\rvert^2+\lvert\nabla\zeta\rvert^2)\nonumber\\
\geq &(c-C\delta)(\lvert\nabla\varphi\rvert^2+\lvert{\rm div}\psi\rvert^2+\lvert\nabla\zeta\rvert^2).
\end{align}
Recall \eqref{C21} and \eqref{C26} for the definition of $\mathcal{H}_0$ and $\mathcal{H}_1$. In the same way, we have
\begin{align}
 I_2&\geq 0,\label{C33}\\
I_3&\geq c\beta\Vert(\nabla\varphi,{\rm div}\psi,\nabla\zeta)\Vert_{\varepsilon-1,\beta}^2, \label{C34}\\
 I_4&\geq 0,\label{C35}
\end{align}
where we let $\delta$ small enough and use the boundary condition $\sigma(t,0)=0$.

Next, we estimate $I_1$ and $I_5$.  Note that the estimates of $I_1$ and $I_5$ will be more complex where we take full advantage of properties of the stationary solution in \eqref{C13} to capture the full energy dissipation of all the zero-order components with positive coefficient. Using Lemma \ref{Clem12}, \eqref{C13}, \eqref{C10c} and the identity $\tilde{n}\tilde{u}\equiv u_\infty$, we get
\begin{align}\label{C36}
I_1\geq \int_{\mathbb{R}_+^N}\varepsilon\beta W_{\varepsilon-1,\beta}&\left\{\frac{1-G(x_1)^{-2}}{2}(1+RT_\infty)\lvert u_\infty\rvert\varphi^2-(1-2G(x_1)^{-2})RT_\infty\varphi\psi_1+\frac{1-G(x_1)^{-2}}{2}m\lvert
u_\infty\rvert\psi_1^2\right.\nonumber\\
&\left.-R(1-2G(x_1)^{-2})\psi_1\zeta+(1-G(x_1)^{-2})\frac{R\lvert u_\infty\rvert}{2(\gamma-1)T_\infty}\zeta^2+(1-2G(x_1)^{-2})\psi_1\sigma\right\}dx\nonumber\\
&+c\beta\Vert\psi'\Vert_{\varepsilon-1,\beta}^2\nonumber\\
&-CN_{\lambda,\beta}(T)\Vert(\vartheta,\sigma)\Vert_{\varepsilon-3,\beta}^2-C\phi_b\int_{\mathbb{R}_+^N}\beta W_{\varepsilon-1,\beta}G(x_1)^{-2}(\lvert\vartheta\rvert^2+\sigma^2)dx,
\end{align}
where $\psi':=(\psi_2,...,\psi_N)$ and $G(x_1)$ is defined in \eqref{C14}. At the same time, we can obtain
\begin{align}\label{C37}
I_5\geq \int_{\mathbb{R}_+^N} W_{\varepsilon,\beta}G(x_1)^{-3}\Gamma\lvert u_\infty\rvert &\left\{(1+R\gamma T_\infty)\varphi^2+\frac{2(1-R\gamma T_\infty)}{\lvert u_\infty\rvert}\varphi\psi_1+3m\psi_1^2+\frac{R\gamma}{(\gamma-1)T_\infty}\zeta^2\right.\nonumber\\
&\left.+\frac{4}{\lvert u_\infty\rvert}\sigma\psi_1+m\lvert\psi'\rvert^2\right\}dx\nonumber\\
&-C(N_{\lambda,\beta}(T)+\phi_b)\int_{\mathbb{R}_+^N}W_{\varepsilon,\beta}G(x_1)^{-3}(\lvert\vartheta\rvert^2+\sigma^2)dx,
\end{align}
where $\Gamma$ is defined in \eqref{C13}. Adding \eqref{C36} to \eqref{C37} together and using the Cauchy-Schwarz inequality
$$\sigma\psi_1\geq-\left(\frac{\lvert u_\infty \rvert}{2}\lvert\sigma\rvert^2+\frac{1}{2\lvert u_\infty \rvert}\psi_1^2\right),$$
one has
\begin{align}\label{C38}
I_1+I_5\geq& I_6+c\beta\Vert \psi'\Vert_{\varepsilon-1,\beta}^2-CN_{\lambda,\beta}(T)\Vert(\vartheta,\sigma)\Vert_{\varepsilon-3,\beta}^2\nonumber\\
&-C\phi_b\int_{\mathbb{R}_+^N}\beta W_{\varepsilon-1,\beta}G(x_1)^{-2}(\lvert \vartheta \rvert^2+\sigma^2)dx-C(N_{\lambda,\beta}(T)+\phi_b)
\int_{\mathbb{R}_+^N}W_{\varepsilon,\beta}G(x_1)^{-3}(\lvert \vartheta \rvert^2+\sigma^2)dx,
\end{align}
where
\begin{align}\label{C39}
I_6=&\int_{\mathbb{R}_+^N}\left\{\frac{\varepsilon\beta}{2}W_{\varepsilon-1,\beta}(1-G(x_1)^{-2})(1+RT_\infty)\lvert u_\infty\rvert+
\Gamma\lvert u_\infty\rvert(1+R\gamma T_\infty)W_{\varepsilon,\beta}G(x_1)^{-3}\right\}\varphi^2dx\nonumber\\
&+\int_{\mathbb{R}_+^N}\left\{-RT_\infty\varepsilon\beta W_{\varepsilon-1,\beta}(1-2G(x_1)^{-2})+2\Gamma(1-R\gamma T_\infty)W_{\varepsilon,\beta}G(x_1)^{-3}\right\}\varphi\psi_1dx\nonumber\\
&+\frac{1}{\lvert u_\infty\rvert}\int_{\mathbb{R}_+^N}\left\{\frac{\varepsilon\beta}{2}W_{\varepsilon-1,\beta}
\left[\gamma RT_\infty+(1-\gamma RT_\infty)G(x_1)^{-2}\right]+
\Gamma(1+3\gamma RT_\infty)W_{\varepsilon,\beta}G(x_1)^{-3}\right\}\psi_1^2dx\nonumber\\
&-\frac{\lvert u_\infty\rvert}{2}\int_{\mathbb{R}_+^N}\left\{\varepsilon\beta W_{\varepsilon-1,\beta}(1-2G(x_1)^{-2})+4\Gamma W_{\varepsilon,\beta}G(x_1)^{-3}\right\}\sigma^2dx\nonumber\\
&+\int_{\mathbb{R}_+^N}\left\{\frac{\varepsilon\beta}{2}W_{\varepsilon-1,\beta}\frac{R\lvert u_\infty\rvert}{(\gamma-1)T_\infty}(1-G(x_1)^{-2})
+\frac{\Gamma \gamma R\lvert u_\infty\rvert}{(\gamma-1)T_\infty}W_{\varepsilon,\beta}G(x_1)^{-3}\right\}\zeta^2dx\nonumber\\
&-\int_{\mathbb{R}_+^N}\varepsilon\beta W_{\varepsilon-1,\beta}R(1-2G(x_1)^{-2})\psi_1\zeta dx.
\end{align}

$Claim$
\begin{align}\label{C40}
I_6\geq c\beta^3\Vert(\varphi,\psi_1,\zeta)\Vert_{\varepsilon-3,\beta}^2+c\beta\Vert\nabla\sigma\Vert_{\varepsilon-1,\beta}^2.
\end{align}
In fact, we multiply $\eqref{C28}$ by $-\varepsilon\beta\sigma W_{\varepsilon-1,\beta}$ and integrate the result over $\mathbb{R}_+^N$ with the help of Lemma \ref{Clem11}, \eqref{C13} and the Schwarz inequality to get
\begin{align}\label{C41}
\int_{\mathbb{R}_+^N} &\varepsilon\beta W_{\varepsilon-1,\beta}\left\{\lvert\nabla\sigma\rvert^2+\frac{1}{2}(1-2G(x_1)^{-2})\sigma^2\right\}dx\nonumber\\
\leq &\int_{\mathbb{R}_+^N}\frac{\varepsilon\beta}{2}W_{\varepsilon-1,\beta}(1-2G(x_1)^{-2})\varphi^2dx+\int_{\mathbb{R}_+^N}\frac{1}{2}\varepsilon
(\varepsilon-1)(\varepsilon-2)\beta^3W_{\varepsilon-3,\beta}\varphi^2dx\nonumber\\
&+C(\phi_b+\beta^2+N_{\lambda,\beta}(T)\beta^{-2})\beta^3\Vert\varphi\Vert_{\varepsilon-3,\beta}^2.
\end{align}
Applying the above estimate into \eqref{C39}, we have
\begin{align}\label{C42}
I_6\geq\int_{\mathbb{R}_+^N}\beta W_{\varepsilon-1,\beta}Q(x)dx+\lvert u_\infty\rvert\varepsilon\beta\Vert\nabla\sigma\Vert_{\varepsilon-1,\beta}^2
-C(\phi_b+\beta^2+N_{\lambda,\beta}(T)\beta^{-2})\beta^3\Vert\varphi\Vert_{\varepsilon-3,\beta}^2,
\end{align}
where the quadratic form of $\varphi$, $\psi_1$, and $\zeta$ is defined by
\begin{align}\label{C43}
Q(x)=\lvert u_\infty\rvert q_1(x_1)\varphi^2+q_2(x_1)\varphi\psi_1+\frac{1}{\lvert u_\infty\rvert}q_3(x_1)\psi_1^2+\lvert u_\infty\rvert q_4(x_1)\zeta^2+q_5(x_1)\zeta\psi_1
\end{align}
with
\begin{align*}
q_1(x_1)&=\frac{\varepsilon}{2}RT_\infty+B(x_1)^{-2}\Gamma^{-2}\left\{\frac{(1-RT_\infty)\varepsilon}{2}S(x_1)^2-\frac{\Gamma^2}{2}\varepsilon(\varepsilon-1)(\varepsilon
-2)+(\gamma RT_\infty-1)S(x_1)^3\right\},\\
q_2(x_1)&=-RT_\infty\varepsilon+B(x_1)^{-2}\Gamma^{-2}\left\{2\varepsilon RT_\infty S(x_1)^2+2(1-\gamma RT_\infty) S(x_1)^3\right\},\\
q_3(x_1)&=\frac{\varepsilon}{2}\gamma RT_\infty+B(x_1)^{-2}\Gamma^{-2}\left\{\frac{(1-RT_\infty)\varepsilon}{2}S(x_1)^2+(3\gamma RT_\infty+3)S(x_1)^3\right\},\\
q_4(x_1)&=\frac{\varepsilon R}{2(\gamma-1)T_\infty}+B(x_1)^{-2}\Gamma^{-2}\left\{-\frac{\varepsilon R}{2(\gamma-1)T_\infty}S(x_1)^2+\frac{\gamma R}{(\gamma-1)T_\infty}S(x_1)^3\right\},\\
q_5(x_1)&=-\varepsilon R+2\varepsilon RB(x_1)^{-2}\Gamma^{-2}S(x_1)^2,\\
B(x_1)&=x_1+\beta^{-1}, \ \ \ S(x_1)=(x_1+\beta^{-1})/(x_1+\Gamma^{-1}\phi_b^{-1/2}).
\end{align*}
From \eqref{C97}$-$\eqref{C101}, it holds that
\begin{align*}
S(x_1)\geq1,\quad B(x_1)^{-2}\leq \beta^2\leq C\phi_b \leq C\delta.
\end{align*}
Then direct computations and the above estimate indicate that
\begin{align}\label{C44}
q_1(x_1)>0,\quad q_3(x_1)>0,\quad q_4(x_1)>0,
\end{align}
\begin{align}\label{C45}
q_2(x_1)^2-4q_1(x_1)q_3(x_1)<0,\quad q_5(x_1)^2-4q_3(x_1)q_4(x_1)<0,
\end{align}
and
\begin{align}\label{C46}
&q_1(x_1)q_5(x_1)^2+q_4(x_1)q_2(x_1)^2-4q_1(x_1)q_3(x_1)q_4(x_1)\nonumber\\
\leq&\frac{\varepsilon^2R^2}{2(\gamma-1)}B(x_1)^{-2}\left\{\varepsilon(\varepsilon-1)(\varepsilon-2)-12\left(\frac{2\varepsilon}{\gamma+1}+2\right)+C\beta^2\right\}\nonumber\\
\leq& -cB(x_1)^{-2},
\end{align}
where we have used the definition of $\lambda_0$, the inequality $\varepsilon \leq \lambda <\lambda_0$, and the smallness of $\delta$.  Thus, combining $\eqref{C44}$ together with $\eqref{C45}$ and $\eqref{C46}$, it holds that
\begin{align}\label{C47}
\int_{\mathbb{R}_+^N}\beta W_{\varepsilon-1,\beta}Q(x)dx \geq c\beta\int_{\mathbb{R}_+^N}W_{\varepsilon-1,\beta}B(x_1)^{-2}(\varphi^2+\psi_1^2+\zeta^2)dx=c\beta^3\Vert(\varphi,\psi_1,\zeta)\Vert_{\varepsilon-3,\beta}^2.
\end{align}
Substituting \eqref{C47} into \eqref{C42}, then \eqref{C40} is an easy consequence of letting $\beta^2\leq C\phi_b\leq C\delta$ and $N_{\lambda,\beta}(T)/\beta^3\leq \delta$ for $\delta >0$ small enough.

For the last three terms on the right of \eqref{C38}, it holds that
\begin{align}\label{C48}
C&N_{\lambda,\beta}(T)\Vert(\vartheta,\sigma)\Vert_{\varepsilon-3,\beta}^2+C\phi_b\int_{\mathbb{R}_+^N}\beta W_{\varepsilon-1,\beta}G(x_1)^{-2}(\lvert \vartheta \rvert^2+\sigma^2)dx\nonumber\\
&+C(N_{\lambda,\beta}(T)+\phi_b)
\int_{\mathbb{R}_+^N}W_{\varepsilon,\beta}G(x_1)^{-3}(\lvert \vartheta \rvert^2+\sigma^2)dx
\leq C\delta\beta^3\Vert\vartheta\Vert_{\varepsilon-3,\beta}^2.
\end{align}
By substituting \eqref{C48} and \eqref{C40} into \eqref{C38}, we have
\begin{align}\label{C49}
I_1+I_5\geq (c-C\delta)(\beta^3\Vert\vartheta\Vert_{\varepsilon-3,\beta}^2+\beta\Vert\nabla\sigma\Vert_{\varepsilon-1,\beta}^2).
\end{align}

Finally, we estimate the last term of \eqref{C31}. Recall \eqref{C30}, \eqref{C23}, and \eqref{C27} for the definition of $\mathcal{N}_1$, $\mathcal{R}_0$, and $\mathcal{R}_1$ respectively. From \eqref{C13}, the Sobolev inequality, the Cauchy-Schwarz inequality and the elliptic estimate in Lemma \ref{Clem11}, we can easily get
\begin{align}\label{C50}
&\left\lvert\int_{\mathbb{R}_+^N}W_{\varepsilon,\beta}e^{-\tilde{\phi}}\mathcal{R}_0dx\right\rvert\nonumber\\
\leq &C\Vert(1+\beta x_1)^2\vartheta\Vert_{L^\infty}\int_{\mathbb{R}_+^N}W_{\varepsilon-2,\beta}\lvert\vartheta\rvert \lvert\nabla\vartheta\rvert dx+C\Vert\vartheta\Vert_{L^\infty}\int_{\mathbb{R}_+^N}W_{\varepsilon,\beta}G(x_1)^{-3}\lvert\vartheta\rvert^2dx\nonumber\\
\leq &CN_{\lambda,\beta}(T)(\Vert\vartheta\Vert_{\varepsilon-3,\beta}^2+\Vert\nabla\vartheta\Vert_{\varepsilon-1,\beta}^2)\nonumber\\
\leq &C\delta(\beta^3\Vert\vartheta\Vert_{\varepsilon-3,\beta}^2+\beta\Vert\nabla\vartheta\Vert_{\varepsilon-1,\beta}^2)
\end{align}
and
\begin{align}\label{C51}
&\left\lvert\int_{\mathbb{R}_+^N}W_{\varepsilon,\beta}\mathcal{R}_1dx\right\rvert\nonumber\\
\leq &C\Vert(1+\beta x_1)\nabla\vartheta\Vert_{L^\infty}\int_{\mathbb{R}_+^N}W_{\varepsilon-1,\beta}\lvert\nabla\vartheta\rvert^2 dx
+C\int_{\mathbb{R}_+^N}W_{\varepsilon,\beta}G(x_1)^{-3}\lvert\nabla\vartheta\rvert^2dx
+C\int_{\mathbb{R}_+^N}W_{\varepsilon,\beta}G(x_1)^{-4}\lvert\nabla\vartheta\rvert\lvert\vartheta\rvert dx\nonumber\\
\leq &C(N_{\lambda,\beta}(T)+\phi_b\beta)\Vert\nabla\vartheta\Vert_{\varepsilon-1,\beta}^2+C\phi_b\beta^3\Vert\vartheta\Vert_{\varepsilon-3,\beta}^2\nonumber\\
\leq &C\delta(\beta^3\Vert\vartheta\Vert_{\varepsilon-3,\beta}^2+\beta\Vert\nabla\vartheta\Vert_{\varepsilon-1,\beta}^2),
\end{align}
where we have used the following two facts:
\begin{align*}
\int_{\mathbb{R}_+^N}W_{\varepsilon,\beta}G(x_1)^{-3}\lvert\nabla\vartheta\rvert^2dx
&\leq C\phi_b\int_{\mathbb{R}_+^N}W_{\varepsilon,\beta}G(x_1)^{-1}\lvert\nabla\vartheta\rvert^2dx\\
&\leq C\phi_b\beta\Vert\nabla\vartheta\Vert_{\varepsilon-1,\beta}^2
\end{align*}
and
\begin{align*}
\int_{\mathbb{R}_+^N}W_{\varepsilon,\beta}G(x_1)^{-4}\lvert\nabla\vartheta\rvert\lvert\vartheta\rvert dx
&\leq C\phi_b\int_{\mathbb{R}_+^N}W_{\varepsilon,\beta}G(x_1)^{-2}\lvert\nabla\vartheta\rvert\lvert\vartheta\rvert dx\\
&\leq C\phi_b\int_{\mathbb{R}_+^N}W_{\varepsilon,\beta}G(x_1)^{-3}\lvert\vartheta\rvert dx+C\phi_b\int_{\mathbb{R}_+^N}W_{\varepsilon,\beta}G(x_1)^{-1}\lvert\nabla\vartheta\rvert dx\\
&\leq C\phi_b\beta^3\Vert\vartheta\Vert_{\varepsilon-3,\beta}^2+C\phi_b\beta\Vert\nabla\vartheta\Vert_{\varepsilon-1,\beta}^2.
\end{align*}
Other terms are similar to $\eqref{C50}-\eqref{C51}$. Consequently, we estimate the last term of \eqref{C31}, that is
\begin{align}\label{C52}
\left\lvert\int_{\mathbb{R}_+^N}W_{\varepsilon,\beta}\mathcal{N}_1dx\right\rvert
\leq C\delta(\beta^3\Vert\vartheta\Vert_{\varepsilon-3,\beta}^2+\beta\Vert\nabla\vartheta\Vert_{\varepsilon-1,\beta}^2).
\end{align}

Substituting $\eqref{C33}-\eqref{C35}$, $\eqref{C49}$, and $\eqref{C52}$ into $\eqref{C31}$, we have
\begin{align}\label{C53}
\frac{d}{dt}\int_{\mathbb{R}_+^N}W_{\varepsilon,\beta}\left(e^{-\tilde{\phi}}\mathcal{E}_0+\mathcal{E}_1+\frac{1}{2}\tilde{n}^2\varphi^2\right)dx+\beta^3\Vert\vartheta\Vert_{\varepsilon-1,\beta}^2+\beta\Vert(\nabla\varphi, {\rm div}\psi, \nabla\zeta, \nabla\sigma)\Vert_{\varepsilon-1,\beta}^2
\leq C\delta\beta\Vert\nabla\psi\Vert_{\varepsilon-1,\beta}^2,
\end{align}
provided that $\delta>0$ is sufficiently small, where $\mathcal{E}_0$ and $\mathcal{E}_1$ are defined in \eqref{C20} and \eqref{C25} respectively. Multiply \eqref{C53} by $(1+\beta\tau)^\xi$ and integrate over $(0,t)$ to complete the proof of Lemma \ref{Clem3}.
\end{proof}

\begin{lem}\label{Clem4}
Under the same assumption as in Proposition \ref{CProp1}, it holds for any $\xi\geq0$ and $t\in [0,T]$ that
\begin{align}\label{C103}
(1+&\beta t)^\xi\Vert\vartheta(t)\Vert_{\varepsilon,\beta,1}^2+\int_0^t(1+\beta\tau)^\xi\left\{\beta^3\Vert\vartheta(\tau)\Vert_{\varepsilon-3,\beta}^2
+\beta\Vert\nabla\vartheta(\tau)\Vert_{\varepsilon-1,\beta}^2\right\}d\tau\nonumber\\
\leq &C\Vert\vartheta_0\Vert_{\varepsilon,\beta,1}^2+C\xi\int_0^t(1+\beta\tau)^{\xi-1}\beta\Vert\vartheta(\tau)\Vert_{\varepsilon,\beta,1}^2d\tau\nonumber\\
&+C\int_0^t(1+\beta\tau)^\xi\beta\Vert\nabla\sigma(\tau)\Vert_{\varepsilon-1,\beta}^2d\tau+C\int_0^t(1+\beta \tau)^\xi\int_{\mathbb{R}^{N-1}}\sigma_{x_1}^2(\tau,0,x')dx'd\tau,
\end{align}
where positive constants $\delta$ and $C$ are independent of $T$.
\end{lem}
\begin{proof}
Multiplying $\nabla(RT\eqref{C15}_1)$, $\nabla\eqref{C15}_2$, and $\nabla\left(\frac{R}{(\gamma-1)T}\eqref{C15}_3\right)$ by $\tilde{n}\nabla\varphi$, $\tilde{n}\nabla\psi$, and $\tilde{n}\nabla\zeta$ respectively, one has that
\begin{align}\label{C54}
(\widetilde{\mathcal{E}}_1)_t+{\rm div}\widetilde{\mathcal{H}}_1-\tilde{n}{\rm div}\psi\Delta\sigma=\widetilde{\mathcal{R}}_1,
\end{align}
where we have denoted
\begin{align}\label{C55}
\widetilde{\mathcal{E}}_1=\frac{\tilde{n}}{2}RT\lvert\nabla\varphi\rvert^2+\frac{\tilde{n}}{2}m\lvert\nabla\psi\rvert^2+\frac{\tilde{n}R}{2(\gamma-1)T}\lvert\nabla\zeta\rvert^2,
\end{align}
\begin{align}\label{C56}
\widetilde{\mathcal{H}}_1=\frac{\tilde{n}}{2}RTu\lvert\nabla\varphi\rvert^2+\tilde{n}RT\nabla\varphi\cdot\nabla\psi+\frac{\tilde{n}}{2}mu\lvert \nabla\psi\rvert^2+R\tilde{n}\nabla\psi\cdot\nabla\zeta+\frac{R\tilde{n}u}{2(\gamma-1)T}\lvert\nabla\zeta\rvert^2+\tilde{n}{\rm div}\psi\nabla\sigma-\tilde{n}\nabla\psi\cdot\nabla\sigma,
\end{align}
and
\begin{align}\label{C57}
\widetilde{\mathcal{R}}_1=&\left[\frac{\tilde{n}R}{2}\zeta_t+\frac{RTu_1}{2}\tilde{n}_{x_1}+\frac{\tilde{n}R}{2}{\rm div}(Tu)\right]\lvert\nabla\varphi\rvert^2
+\left[\frac{mu_1}{2}\tilde{n}_{x_1}+\frac{m\tilde{n}}{2}{\rm div}u\right]\lvert\nabla\psi\rvert^2\nonumber\\
&+\left[\frac{R\tilde{n}}{2(\gamma-1)T}{\rm div}u+\frac{Ru_1}{2(\gamma-1)T}\tilde{n}_{x_1}-\frac{R\tilde{n}u}{2(\gamma-1)T^2}\nabla T-
\frac{R\tilde{n}}{2(\gamma-1)T^2}\zeta_t-\frac{R\tilde{n}}{T}\tilde{u}_{x_1}\right]\lvert\nabla\zeta\rvert^2\nonumber\\
&-\tilde{n}\tilde{u}_{x_1}\left[RT(\partial_{x_1}\varphi)^2+m\lvert\nabla\psi_1\rvert^2+m(\partial_{x_1}\psi)^2+\frac{R}{(\gamma-1)T}\lvert\partial_{x_1}\zeta\rvert^2+
\frac{R}{T}\lvert\nabla\zeta\rvert^2\right]\nonumber\\
&+\tilde{n}_{x_1}(\partial_{x_1}\sigma{\rm div}\psi-\nabla\psi_1\cdot\nabla\sigma)\nonumber\\
&-\tilde{n}\tilde{T}_{x_1}\left[Ru\partial_{x_1}\varphi\cdot\nabla\varphi-R\nabla\psi_1\cdot\nabla\varphi+R\nabla\varphi\cdot\partial_{x_1}\psi-\frac{Ru}{(\gamma-1)T^2}
\partial_{x_1}\zeta\cdot\nabla\zeta+\frac{R}{(\gamma-1)T}\nabla\psi_1\cdot\nabla\zeta\right]\nonumber\\
&-R\tilde{n}T\sum_{i,j=1}^3\partial_{x_i}\varphi\partial_{x_j}\varphi\partial_{x_i}\psi_j-R\tilde{n}(u\cdot\nabla\varphi)(\nabla\varphi\cdot\nabla\zeta)
-m\tilde{n}\sum_{i,j,l=1}^3\partial_{x_j}\psi_i\partial_{x_i}\psi_l\partial_{x_j}\psi_l+R\tilde{n}\nabla\varphi\cdot\nabla\psi\cdot\nabla\zeta\nonumber\\
&-R\tilde{n}\sum_{i,j=1}^3\partial_{x_j}\zeta\partial_{x_i}\varphi\partial_{x_j}\zeta
-\frac{R\tilde{n}}{(\gamma-1)T}\nabla\zeta\cdot\nabla\psi\cdot\nabla\zeta+\frac{R\tilde{n}}{(\gamma-1)T^2}u\cdot\nabla\zeta\lvert\nabla\zeta\rvert^2-R\tilde{n}\nabla T({\rm div}\psi+\varphi_t)\nabla\varphi\nonumber\\
&+\frac{R\tilde{n}}{(\gamma-1)T^2}\nabla T\cdot\nabla\zeta\zeta_t-R\tilde{n}_{x_1}\psi_1\nabla\varphi\cdot
\nabla T+\frac{R\tilde{n}\tilde{T}_{x_1}}{(\gamma-1)T^2}\psi_1\nabla \zeta\cdot\nabla T+\frac{R\tilde{n}\tilde{u}_{x_1}}{(\gamma-1)T^2}\zeta\nabla \zeta\cdot\nabla T\nonumber\\
&-R\tilde{n}\tilde{v}_{x_1x_1}(T\psi_1\partial_{x_1}\varphi+\zeta\partial_{x_1}\psi_1)-\tilde{n}\tilde{u}_{x_1x_1}\left(m\psi_1\partial_{x_1}\psi_1+\frac{R\tilde{n}}{T}\zeta
\partial_{x_1}\zeta\right)-\frac{R}{(\gamma-1)T}\tilde{T}_{x_1x_1}\psi_1\partial_{x_1}\zeta.
\end{align}
Then we multiply \eqref{C19} by $e^{-\tilde{\phi}}$ and add the result to \eqref{C54} with the help of \eqref{C29} to get
\begin{align}\label{C58}
&\left(e^{-\tilde{\phi}}\mathcal{E}_0+\widetilde{\mathcal{E}}_1+\frac{1}{2}\tilde{n}^2\varphi^2\right)_t+{\rm div}\left(e^{-\tilde{\phi}}\mathcal{H}_0+\widetilde{\mathcal{H}}_1+\frac{\tilde{n}^2u\varphi^2}{2}\right)+e^{-\tilde{\phi}}\tilde{\phi}_{x_1}(\mathcal{H}_0)_1
+e^{-\tilde{\phi}}\mathcal{D}_0\nonumber\\
&-\left(\frac{1}{2}\tilde{n}^2\tilde{u}_{x_1}+\tilde{n}u_1\tilde{n}_{x_1}\right)\varphi^2+\tilde{n}^2\tilde{v}_{x_1}\varphi\psi_1\nonumber\\
=&e^{-\tilde{\phi}}\mathcal{R}_0+\widetilde{\mathcal{R}}_1+\frac{1}{2}\tilde{n}^2{\rm div}\psi\varphi^2+\frac{\tilde{n}}{2}\left(\tilde{n}e^{\theta_1\varphi}\varphi^2
-e^{-(\theta_2\sigma+\tilde{\phi})}\sigma^2\right){\rm div}\psi
:=\widetilde{\mathcal{N}}_1.
\end{align}

Multiply \eqref{C58} by $W_{\varepsilon,\beta}=(1+\beta x_1)^\varepsilon$ and integrating the resulting equation over $\mathbb{R}_+^N$ to obtain
\begin{align}\label{C59}
\frac{d}{dt}&\int_{\mathbb{R}_+^N}W_{\varepsilon,\beta}\left(e^{-\tilde{\phi}}\mathcal{E}_0+\widetilde{\mathcal{E}}_1+\frac{1}{2}\tilde{n}^2\varphi^2\right)dx\nonumber\\
&+I_1+I_2+I_5
+\underbrace{\int_{\mathbb{R}_+^N}\varepsilon\beta W_{\varepsilon-1,\beta}(-\widetilde{\mathcal{H}}_1)_1dx}_{I_7}+\underbrace{\int_{x_1=0}(-\widetilde{\mathcal{H}}_1)_1dx'}_{I_8}\nonumber\\
=&\int_{\mathbb{R}_+^N}W_{\varepsilon,\beta}\widetilde{\mathcal{N}}_1dx,
\end{align}
where $(\widetilde{\mathcal{H}}_1)_1$ also denotes the first component of $\widetilde{\mathcal{H}}_1$ and $I_1, I_2, I_5$ are defined in  \eqref{C31}.

It follows from the same computations in \eqref{C32} and the Cauchy-Schwarz inequality that
\begin{align}\label{C60}
I_7+I_8\geq c\beta\Vert\nabla\vartheta\Vert_{\varepsilon-1,\beta}^2-C\beta\Vert\nabla\sigma\Vert_{\varepsilon-1,\beta}^2-C\int_{\mathbb{R}^{N-1}}\sigma_{x_1}^2(t,0,x')dx'.
\end{align}
In a manner similar to the derivation of \eqref{C52}, we get
\begin{align}\label{C61}
\left\lvert\int_{\mathbb{R}_+^N}W_{\varepsilon,\beta}\widetilde{\mathcal{N}}_1dx\right\rvert
\leq C\delta(\beta^3\Vert\vartheta\Vert_{\varepsilon-3,\beta}^2+\beta\Vert(\nabla\vartheta, \nabla\sigma)\Vert_{\varepsilon-1,\beta}^2).
\end{align}
Substituting $\eqref{C33}$, $\eqref{C49}$, $\eqref{C60}$, and $\eqref{C61}$  into $\eqref{C59}$, we have
\begin{align}\label{C62}
&\frac{d}{dt}\int_{\mathbb{R}_+^N}W_{\varepsilon,\beta}\left(e^{-\tilde{\phi}}\mathcal{E}_0+\widetilde{\mathcal{E}}_1+\frac{1}{2}\tilde{n}^2\varphi^2\right)dx
+\beta^3\Vert\vartheta\Vert_{\varepsilon-3,\beta}^2+\beta\Vert\nabla\vartheta\Vert_{\varepsilon-1,\beta}^2\nonumber\\
\leq& C\delta\beta\Vert\nabla\sigma\Vert_{\varepsilon-1,\beta}^2+C\int_{\mathbb{R}^{N-1}}\sigma_{x_1}^2(t,0,x')dx',
\end{align}
provided that $\delta>0$ is sufficiently small, where $\mathcal{E}_0$ and $\widetilde{\mathcal{E}}_1$ are defined in \eqref{C20} and \eqref{C55} respectively. Multiply \eqref{C62} by $(1+\beta\tau)^\xi$ and integrate over $(0,t)$ to give the desired estimate \eqref{C103}.
\end{proof}

\begin{lem}\label{Clem5}
Under the same assumption as in Proposition \ref{CProp1}, it holds for any $\xi\geq0$ and $t\in [0,T]$ that
\begin{align}\label{C104}
(1+&\beta t)^\xi\Vert\vartheta(t)\Vert_{\varepsilon,\beta}^2+\int_0^t(1+\beta \tau)^\xi\int_{\mathbb{R}^{N-1}}\sigma_{x_1}^2(\tau,0,x')dx'd\tau\nonumber\\
\leq &C\Vert\vartheta_0\Vert_{\varepsilon,\beta,1}^2+C\xi\int_0^t(1+\beta\tau)^{\xi-1}\beta\Vert\vartheta(\tau)\Vert_{\varepsilon,\beta,1}^2d\tau\nonumber\\
&+C\int_0^t(1+\beta\tau)^\xi\left\{\beta^3\Vert\vartheta(\tau)\Vert_{\varepsilon-3,\beta}^2+\beta\Vert(\nabla\varphi,{\rm div}\psi,\nabla\zeta,\nabla\sigma)(\tau)\Vert_{\varepsilon-1,\beta}^2\right\}d\tau,
\end{align}
where positive constants $\delta$ and $C$ are independent of $T$.
\end{lem}
\begin{proof}
Multiplying $e^{-v}\eqref{C15}_{4t}$ by $\sigma$, using $\eqref{C15}_1-\eqref{C15}_3$, one has
\begin{align}\label{C63}
&\left(\frac{1}{2}e^{-\phi-v}\lvert\sigma\rvert^2+\frac{1}{2}e^{-v}\lvert\nabla\sigma\rvert^2+\frac{\mathcal{E}_0}{\tilde{n}}\right)_{t}
+{\rm div}\left(\frac{\mathcal{H}_0}{\tilde{n}}-e^{-v}\nabla\sigma_t\sigma-\varphi\psi\sigma-\tilde{U}\varphi\sigma\right)\nonumber\\
&=\frac{\mathcal{R}_0}{\tilde{n}}-\frac{e^{-v}}{2}\left[\varphi_t\lvert\nabla\sigma\rvert^2+e^{-\phi}(\varphi_t+\sigma_t)\sigma^2\right]
+e^{-v}\sigma\nabla\varphi\cdot\nabla\sigma_t-{\rm div}\psi\sigma\varphi-\varphi\psi\cdot\nabla\sigma\nonumber\\
&\quad+\tilde{v}_{x_1}(e^{-v}\sigma\sigma_{tx_1}-R\zeta\psi_1-RT\varphi\psi_1+\psi_1\sigma)+\tilde{u}_{x_1}\left[\frac{RT}{2}\varphi^2
+\left(\frac{R}{2(\gamma-1)T}-\frac{R}{T}\right)\zeta^2+\frac{m}{2}
\psi^2-m\psi_1^2-\sigma\varphi\right]\nonumber\\
&\quad+\tilde{T}_{x_1}\left[\frac{Ru_1}{2}\varphi^2+R\varphi\psi_1-\frac{R}{(\gamma-1)T}\psi_1\zeta-\frac{Ru_1}{(\gamma-1)T^2}\zeta^2\right]
+\tilde{u}\sigma_{x_1}(e^{\varphi}-1-\varphi)-\tilde{u}\sigma_{x_1}(e^{\varphi}-1)\nonumber\\
&=:\mathcal{N}_2-\tilde{u}\sigma_{x_1}(e^{\varphi}-1),
\end{align}
where $\mathcal{E}_0$, $\mathcal{H}_0$, and $\mathcal{R}_0$ are defined in \eqref{C20}, \eqref{C21}, and \eqref{C23} respectively.
Then, we multiply \eqref{C63} by $W_{\varepsilon,\beta}=(1+\beta x_1)^\varepsilon$ and integrating the resulting equation over $\mathbb{R}_+^N$ to get
\begin{align}\label{C64}
\frac{d}{dt}&\int_{\mathbb{R}_+^N}W_{\varepsilon,\beta}\left(\frac{1}{2}e^{-\phi-v}\lvert\sigma\rvert^2+\frac{1}{2}e^{-v}\lvert\nabla\sigma\rvert^2+
\frac{\mathcal{E}_0}{\tilde{n}}\right)dx\nonumber\\
&+\underbrace{\int_{\mathbb{R}_+^N}\varepsilon\beta W_{\varepsilon-1,\beta}\left\{\frac{-R\tilde{T}\tilde{u}}{2}\varphi^2-R\tilde{T}\varphi\psi_1+\frac{-m\tilde{u}}{2}\lvert\psi\rvert^2-R\psi_1\zeta+\frac{-R\tilde{u}}{2(\gamma-1)T}\zeta^2
+\sigma\psi_1-\frac{-\tilde{u}}{2}e^{-\tilde{v}-\tilde{\phi}}\sigma^2\right.}_{I_9}\nonumber\\
&\underbrace{\left.+\tilde{u}\varphi\sigma+e^{-v}\sigma\sigma_{tx_1}+\frac{-\tilde{u}e^{-\tilde{v}}}{2}\left(2\sigma_{x_1}^2-\lvert\nabla\sigma\rvert^2\right)\right\}dx}_{I_9}
+\underbrace{\int_{x_1=0}\left\{\left(\frac{-\mathcal{H}_0}{\tilde{n}}\right)_1+\frac{-\tilde{u}e^{-\tilde{v}}}{2}\sigma_{x_1}^2\right\}dx'}_{I_{10}}\nonumber\\
=&\int_{\mathbb{R}_+^N}W_{\varepsilon,\beta}\left\{\mathcal{N}_2-\frac{\partial_{x_1}(\tilde{u}e^{-\tilde{v}})}{2}(\lvert\nabla\sigma\rvert^2-2\sigma_{x_1}^2)
-\tilde{u}\sigma_{x_1}\frac{e^{-\tilde{\phi}}}{\tilde{n}}(e^{-\sigma}-1+\sigma)-\frac{\partial_{x_1}(\tilde{u}e^{-\tilde{v}-\tilde{\phi}})}{2}
\sigma^2\right\}dx\nonumber\\
&+\int_{\mathbb{R}_+^N}\varepsilon\beta W_{\varepsilon-1,\beta}\left\{\frac{R(Tu_1-\tilde{T}\tilde{u})}{2}\varphi^2+R\zeta\varphi\psi_1+
\frac{m\psi_1}{2}\psi^2+\frac{R\zeta^2}{2(\gamma-1)}\left(\frac{u_1}{T}-\frac{\tilde{u}}{\tilde{T}}\right)+\varphi\psi_1\sigma\right\}dx,
\end{align}
where we have used the boundary condition $\sigma(t,0)=0$ and the following key fact:
\begin{align}\label{C65}
\int_{\mathbb{R}_+^N}&W_{\varepsilon,\beta}\tilde{u}\sigma_{x_1}(e^{\varphi}-1)dx\nonumber\\
=&\int_{\mathbb{R}^{N-1}}\frac{-\tilde{u}e^{-\tilde{v}}}{2}\sigma_{x_1}^2dx'+\int_{\mathbb{R}_+^N}\varepsilon\beta W_{\varepsilon-1,\beta}\left\{\frac{-\tilde{u}e^{-\tilde{v}}}{2}\left(2\sigma_{x_1}^2-\lvert\nabla\sigma\rvert^2\right)-\frac{-\tilde{u}}{2}e^{-\tilde{v}-\tilde{\phi}}\sigma^2
\right\}dx\nonumber\\
&+\int_{\mathbb{R}_+^N}W_{\varepsilon,\beta}\left\{\frac{\partial_{x_1}(\tilde{u}e^{-\tilde{v}})}{2}\left(\lvert\nabla\sigma\rvert^2-2\sigma_{x_1}^2\right)
+\tilde{u}\sigma_{x_1}\frac{e^{-\tilde{\phi}}}{\tilde{n}}(e^{-\sigma}-1+\sigma)+\frac{\partial_{x_1}(\tilde{u}e^{-\tilde{v}-\tilde{\phi}})}{2}\sigma^2\right\}dx.
\end{align}
We note that the derivation of the equality \eqref{C65} mainly  uses $\eqref{C15}_4$ and integration by parts.

Similarly to \eqref{C32}, $I_{10}$ can be estimated as
\begin{align}\label{C66}
I_{10}\geq c\int_{\mathbb{R}^{N-1}}\sigma_{x_1}^2(t,0,x')dx'.
\end{align}

Now let's focus on the estimate of $I_9$. From \eqref{C28}, \eqref{C13},  integration by parts, the eighth term in $I_9$ is given by
\begin{align}\label{C67}
\int_{\mathbb{R}^N_+}\varepsilon\beta W_{\varepsilon-1,\beta}\tilde{u}\varphi\sigma dx&\geq\int_{\mathbb{R}^N_+}\varepsilon\beta W_{\varepsilon-1,\beta}\frac{-\tilde{u}}{\tilde{n}}
\left(-\Delta\sigma+e^{-\phi}\sigma\right)\sigma dx-CN_{\lambda,\beta}(T)\Vert(\varphi,\sigma)\Vert_{\varepsilon-3,\beta}^2\nonumber\\
&\geq \int_{\mathbb{R}^N_+}\varepsilon\beta W_{\varepsilon-1,\beta}\frac{-\tilde{u}}{\tilde{n}}
\left(e^{-\tilde{\phi}}\sigma^2+\lvert\nabla\sigma\rvert^2\right) dx-C(N_{\lambda,\beta}(T)+\beta^3)\Vert\varphi\Vert_{\varepsilon-3,\beta}^2.
\end{align}
Also the ninth term in $I_9$ can be rewritten as
\begin{align}\label{C68}
\int_{\mathbb{R}^N_+}\varepsilon\beta W_{\varepsilon-1,\beta}e^{v}\sigma \sigma_{x_1t}dx
=&-\int_{\mathbb{R}^N_+}\varepsilon\beta W_{\varepsilon-1,\beta}e^{-v}\sigma_t\sigma_{x_1}dx-\int_{\mathbb{R}^N_+}\varepsilon(\varepsilon-1)\beta^2 W_{\varepsilon-2,\beta}e^{-v}\sigma_{t}\sigma dx\nonumber\\
&+\int_{\mathbb{R}^N_+}\varepsilon\beta W_{\varepsilon-1,\beta}v_xe^{-v}\sigma\sigma_tdx,
\end{align}
then utilizing the Sobolev inequality, the Cauchy-Schwarz inequality and \eqref{C13}, we get
\begin{align}\label{C69}
\left\lvert\int_{\mathbb{R}^N_+}\varepsilon\beta W_{\varepsilon-1,\beta}e^{v}\sigma_{x_1t}\sigma dx\right\rvert\leq C\beta\Vert(\sigma_t,\sigma_{x_1})\Vert_{\varepsilon-1,\beta}^2+C\beta^3\Vert\sigma\Vert_{\varepsilon-3,\beta}^2.
\end{align}

Substituting \eqref{C67} and \eqref{C68} into $I_9$ with help of \eqref{C69} and \eqref{C93},  it follows that
\begin{align}\label{C73}
I_9\geq&\int_{\mathbb{R}_+^N}\varepsilon\beta W_{\varepsilon-1,\beta}\left\{\frac{-R\tilde{T}\tilde{u}}{2}\varphi^2-R\tilde{T}\varphi\psi_1+\frac{-m\tilde{u}}{2}\lvert\psi\rvert^2-R\psi_1\zeta+\frac{-R\tilde{u}}{2(\gamma-1)T}\zeta^2
+\sigma\psi_1+\frac{-\tilde{u}}{2}e^{-\tilde{v}-\tilde{\phi}}\sigma^2\right\}dx\nonumber\\
&+\int_{\mathbb{R}_+^N}\varepsilon\beta W_{\varepsilon-1,\beta}\left\{\frac{-\tilde{u}}{2\tilde{n}}\left(2\sigma_{x_1}^2+\lvert\nabla\sigma\rvert^2\right)\right\}dx
-C\beta\Vert(\sigma_t,\sigma_{x_1})\Vert_{\varepsilon-1,\beta}^2-C\beta^3\Vert\sigma\Vert_{\varepsilon-3,\beta}^2\nonumber\\
\geq&\int_{\mathbb{R}_+^N}\varepsilon\beta W_{\varepsilon-1,\beta}\left\{\frac{RT_\infty\lvert u_\infty\rvert}{2}\varphi^2-RT_\infty\varphi\psi_1+\frac{m\lvert u_\infty\rvert}{2}\psi^2-R\psi_1\zeta+\frac{R\lvert u_\infty\rvert}{2(\gamma-1)T_\infty}\zeta^2
+\sigma\psi_1+\frac{\lvert u_\infty\rvert}{2}\sigma^2\right\}dx\nonumber\\
&-C\beta\Vert(\sigma_t,\sigma_{x_1})\Vert_{\varepsilon-1,\beta}^2-C\beta^3\Vert(\vartheta,\sigma)\Vert_{\varepsilon-3,\beta}^2\nonumber\\
=&\int_{\mathbb{R}_+^N}\varepsilon\beta W_{\varepsilon-1,\beta}\Big\{\frac{RT_{\infty}|u_{\infty}|}{2}(\varphi-\frac{1}{|u_{\infty}|}\psi_1)^{2}
+\frac{(\gamma-1)RT_{\infty}}{2|u_{\infty}|}(\psi_1-\frac{|u_{\infty}|}{(\gamma-1)T_{\infty}}\zeta)^{2}+\frac{|u_{\infty}|}{2}(\sigma+\frac{1}{|u_{\infty}|}\psi_1)^{2}\nonumber\\&+\frac{m\lvert u_\infty\rvert}{2}\lvert\psi'\rvert^2\Big\}dx-C\beta\Vert(\sigma_t,\sigma_{x_1})\Vert_{\varepsilon-1,\beta}^2-C\beta^3\Vert(\vartheta,\sigma)\Vert_{\varepsilon-3,\beta}^2\nonumber\\
\geq& -C\beta^3\Vert(\vartheta,\sigma)\Vert_{\varepsilon-3,\beta}^2-C\beta\Vert(\sigma_t,\sigma_{x_1})\Vert_{\varepsilon-1,\beta}^2\nonumber\\
\geq& -C\beta^3\Vert\vartheta\Vert_{\varepsilon-3,\beta}^2-C\beta\Vert(\nabla\varphi,{\rm div}\psi,\nabla\sigma)\Vert_{\varepsilon-1,\beta}^2,
\end{align}
where we have used the elliptic estimate in Lemma \ref{Clem11} and $\eqref{C15}_1$ in the last inequality.

For the term on the right-hand side of \eqref{C64}, similar to \eqref{C52}, it holds that
\begin{align}\label{C74}
\left\lvert\text{the right-hand side of \eqref{C64}}\right\rvert\leq C(\beta^3+N_{\lambda,\beta}(T))\Vert\vartheta\Vert_{\varepsilon-3,\beta}^2+C(\beta+N_{\lambda,\beta}(T)
)\Vert(\nabla\varphi,{\rm div}\psi,\nabla\zeta,\nabla\sigma)\Vert_{\varepsilon-1,\beta}^2.
\end{align}

Substituting $\eqref{C66}$, $\eqref{C73}$, and $\eqref{C74}$  into $\eqref{C64}$, we have
\begin{align}\label{C75}
&\frac{d}{dt}\int_{\mathbb{R}_+^N}W_{\varepsilon,\beta}\left(\frac{1}{2}e^{-\phi-v}\lvert\sigma\rvert^2+\frac{1}{2}e^{-v}\lvert\nabla\sigma\rvert^2+
\frac{\mathcal{E}_0}{\tilde{n}}\right)dx+\int_{\mathbb{R}^{N-1}}\sigma_{x_1}^2(t,0,x')dx'\nonumber\\
\leq &C\beta^3\Vert\vartheta\Vert_{\varepsilon-3,\beta}^2+C\beta\Vert(\nabla\varphi,{\rm div}\psi,\nabla\zeta,\nabla\sigma)\Vert_{\varepsilon-1,\beta}^2
\end{align}
provided that $\delta>0$ is sufficiently small, where $\mathcal{E}_0$ is defined in \eqref{C20}. The desired estimate \eqref{C104} is obtained by multiplying \eqref{C75} by $(1+\beta\tau)^\xi$ and integrating the resulting inequality over $(0,t)$.
\end{proof}

Combining the results of Lemma \ref{Clem3}$-$Lemma \ref{Clem5}. In other words, multiply \eqref{C103} by $\epsilon^2$ and \eqref{C104} by $\epsilon(>0)$, respectively, and sum up these results and \eqref{C102}. Let $\epsilon$ and $\delta$ suitably small,  we have the following estimate:
\begin{align}\label{C76}
(1+\beta t)^\xi&\Vert\vartheta(t)\Vert_{\varepsilon,\beta,1}^2+\int_0^t(1+\beta\tau)^\xi\left\{\beta^3\Vert\vartheta(\tau)\Vert_{\varepsilon-3,\beta}^2+\beta
\Vert(\nabla\vartheta,\nabla\sigma)(\tau)\Vert_{\varepsilon-1,\beta}^2\right\}d\tau \nonumber\\
&\leq C\Vert\vartheta(0)\Vert_{\varepsilon,\beta,1}^2+C\xi\int_0^t(1+\beta\tau)^{\xi-1}\beta\Vert\vartheta(\tau)\Vert_{\varepsilon,\beta,1}^2d\tau.
\end{align}
Obviously, the equality $\lvert{\rm div}\psi\rvert=\lvert\nabla\psi\rvert$ holds true in the case of $N=1$. We get \eqref{C76} only from Lemma \ref{Clem3}.

Based on the energy method, we need to give the energy estimate of $\Vert\nabla^2\vartheta\Vert_{\varepsilon,\beta,1}$ in the case of $N=2,3$ and the energy estimate of $\Vert\vartheta_{x_1x_1}\Vert_{\varepsilon,\beta}$ in the case of $N=1$ respectively, where $\vartheta=(\varphi,\psi,\zeta)$. In fact, Lemma \ref{Clem14} ensures that we just estimate  $\Vert\partial_{yz}\vartheta\Vert_{\varepsilon,\beta,1}$, where $y$,$z$ are either the time variable $t$ or the spatial variables other than $x_1$ in the case of $N=2,3$ and estimate $\Vert\partial_t(\vartheta,\vartheta_{x_1})\Vert_{\varepsilon,\beta}$ in the case of $N=1$ respectively.

Regarding the case of N=1, we can employ the derivation of  \eqref{C102} to obtain the energy estimate of $\Vert\partial_t(\vartheta,\vartheta_x)\Vert_{\varepsilon,\beta}$. For the sake of brevity, we omit details and only state the results:
\begin{align}\label{C125}
(1+&\beta t)^\xi\Vert\vartheta_t\Vert_{\varepsilon,\beta,1}^2+\int_0^t(1+\beta\tau)^\xi\left\{\beta^3\Vert\vartheta_t(\tau)\Vert_{\varepsilon-3,\beta}^2
+\beta\Vert(\vartheta_{x_1t},\sigma_{x_1t})(\tau)\Vert_{\varepsilon-1,\beta}^2\right\}d\tau\nonumber\\
\leq &C\Vert\vartheta_t(0)\Vert_{\varepsilon,\beta,1}^2+C\xi\int_0^t(1+\beta\tau)^{\xi-1}\Vert\vartheta_t\Vert_{\varepsilon,\beta,1}^2d\tau\nonumber\\
&+C\delta\int_0^t(1+\beta\tau)^\xi\left\{\beta^3\Vert(\vartheta,\vartheta_{x_1x_1})(\tau)\Vert_{\varepsilon-3,\beta}^2
+\beta\Vert(\vartheta_{x_1},\vartheta_{t{x_1}})(\tau)\Vert_{\varepsilon-1,\beta}^2\right\}d\tau.
\end{align}

In what follows, we only focus on the case of $N=2,3$. In order to use the structure of \eqref{C15}$-$\eqref{C17} to drive the estimate of $\Vert\partial_{yz}\vartheta\Vert_{\varepsilon,\beta,1}$, one applies $\partial_{yz}$ to $\eqref{C15}_{1-3}$ and \eqref{C28} to find
\begin{equation}\label{C77}
\left\{
\begin{array}{l}
\partial_{yz}\varphi_t+u\cdot \nabla\partial_{yz}\varphi+{\rm div}\partial_{yz}\psi=-\partial_{yz}\psi_1\tilde{v}_{x_1}-G_1,\\
m(\partial_{yz}\psi_t+u\cdot\nabla\partial_{yz}\psi)+RT\nabla\partial_{yz}\varphi+R\nabla\partial_{yz}\zeta=-m\partial_{yz}\psi_1\tilde{U}_{x_1}-R\partial_{yz}\zeta\nabla\tilde{v}
+\nabla\partial_{yz}\sigma-G_2,\\
\partial_{yz}\zeta_t+u\cdot\nabla\partial_{yz}\zeta+(\gamma-1)T{\rm div}\partial_{yz}\psi=-\partial_{yz}\psi_1\tilde{T}_{x_1}-(\gamma-1)\partial_{yz}\zeta\tilde{u}_{x_1}-G_3,\\
\Delta\partial_{yz}\sigma=\tilde{n}\left[\partial_{yz}\varphi+\partial_{yz}\left(\frac{1}{2}e^{\theta_1\varphi}\varphi^2\right)\right]-e^{-\tilde{\phi}}
\left[-\partial_{yz}\sigma+\partial_{yz}\left(\frac{1}{2}e^{-\theta_2\sigma}\sigma^2\right)\right],\ \ \ \theta_1, \theta_2\in(0,1),
\end{array}
\right.
\end{equation}
where
\begin{align}\label{C78}
&G_1=\partial_{yz}\psi\cdot\nabla\varphi+\partial_z\psi\cdot\nabla\partial_y\varphi+\partial_y\psi\cdot\nabla\partial_z\varphi,\nonumber\\
&G_2=m(\partial_{yz}\psi\cdot\nabla\psi+\partial_z\psi\cdot\nabla\partial_y\psi+\partial_y\psi\cdot\nabla\partial_z\psi)
+R(\partial_{yz}\zeta\cdot\nabla\varphi+\partial_z\zeta\cdot\nabla\partial_y\varphi+\partial_y\zeta\cdot\nabla\partial_z\varphi),\nonumber\\
&G_3=\partial_{yz}\psi\cdot\nabla\zeta+\partial_z\psi\cdot\nabla\partial_y\zeta+\partial_y\psi\cdot\nabla\partial_z\zeta,
\end{align}
and
\begin{align}\label{C137}
\partial_{yz}\sigma(t,0,x')=0,\quad  x'\in \mathbb{R}^{N-1}.
\end{align}

Since the system of \eqref{C77} has the same structure as \eqref{C15}, the energy estimates in Lemma \ref{Clem3}$-$Lemma \ref{Clem5} performed there may be easily modified to yield the desired estimates. For the sake of brevity, we define $\mathcal{E}_0^{yz}$, $\mathcal{E}_1^{yz}$, $\widetilde{\mathcal{E}}_0^{yz}$, $\mathcal{H}_0^{yz}$, $\mathcal{H}_1^{yz}$, $\widetilde{\mathcal{H}}_1^{yz}$, $I_k^{yz}(k=1,...,5)$, and $I_l^{yz}(l=7,...,10)$ by modifying $\mathcal{E}_0$, $\mathcal{E}_1$, $\widetilde{\mathcal{E}}_0$, $\mathcal{H}_0$, $\mathcal{H}_1$, $\widetilde{\mathcal{H}}_1$, $I_k(k=1,...,5)$, and $I_l(l=7,...,10)$ so that $\varphi$, $\psi$, $\zeta$, $\sigma$ therein are replaced with $\partial_{yz}\varphi$, $\partial_{yz}\psi$, $\partial_{yz}\zeta$, $\partial_{yz}\sigma$ respectively. Moreover, we use a notation for convenience:
\begin{align}\label{C79}
\Lambda(t):=\beta\Vert(\nabla\vartheta,\nabla\sigma,\nabla^3\vartheta,\nabla\sigma_{tt})\Vert_{\varepsilon-1,\beta}^2
+\beta^3\Vert(\vartheta,\nabla^2\vartheta)\Vert_{\varepsilon-3,\beta}^2.
\end{align}
\begin{lem}\label{Clem6}
Under the same assumption as in Proposition \ref{CProp1}, it holds for any $\xi\geq0$ and $t\in [0,T]$ that
\begin{align}\label{C105}
(1+\beta t)^\xi&\Vert\partial_{yz}(\vartheta,\nabla\varphi,{\rm div} \psi,\nabla\zeta)(t)\Vert_{\varepsilon,\beta}^2\nonumber\\
&+\int_0^t(1+\beta\tau)^\xi\left\{\beta^3\Vert\partial_{yz}\vartheta(\tau)\Vert_{\varepsilon-3,\beta}^2+\beta\Vert\partial_{yz}(\nabla\varphi,{\rm div}\psi,\nabla\zeta,\nabla\sigma)(\tau)\Vert_{\varepsilon-1,\beta}^2\right\}d\tau\nonumber\\
\leq &C\Vert\partial_{yz}(\vartheta,\nabla\varphi,{\rm div} \psi,\nabla\zeta)(0)\Vert_{\varepsilon,\beta}^2+C\xi\int_0^t(1+\beta\tau)^{\xi-1}\beta\Vert\partial_{yz}(\vartheta,\nabla\varphi,{\rm div}\psi,\nabla\zeta)(\tau)\Vert_{\varepsilon,\beta}^2d\tau\nonumber\\
&+C\delta\int_0^t
(1+\beta\tau)^\xi\Lambda (\tau)d\tau,
\end{align}
where positive constants $\delta$ and $C$ are independent of $T$.
\end{lem}
\begin{proof}
We may argue as the derivation of \eqref{C31} in Lemma \ref{Clem3}. More precisely, multiplying $RT\eqref{C77}_1$, $\eqref{C77}_2$, and $\frac{R}{(\gamma-1)T}\eqref{C77}_3$ by $e^{-\tilde{\phi}}\tilde{n}\partial_{yz}\varphi$, $e^{-\tilde{\phi}}\tilde{n}\partial_{yz}\psi$, and $e^{-\tilde{\phi}}\tilde{n}\partial_{yz}\zeta$ respectively. Then we multiply $\nabla\left(RT\eqref{C77}_1\right)$, ${\rm div}\eqref{C77}_2$, and $\nabla\left(\frac{R}{(\gamma-1)T}\eqref{C77}_3\right)$ by $\tilde{n}\nabla\partial_{yz}\varphi$, $\tilde{n}{\rm div}\partial_{yz}\psi$, and $\tilde{n}\nabla\partial_{yz}\zeta$ respectively. Sum them making use of $\eqref{C77}_4$ and $\eqref{C77}_1$. Multiply the result by $W_{\varepsilon,\beta}=(1+\beta x_1)^\varepsilon$ and integrate over $\mathbb{R}_N^+$. Using Lemma \ref{Clem13}$-$Lemma \ref{Clem15}, we get
\begin{align}\label{C80}
\frac{d}{dt}\int_{\mathbb{R}_+^N}W_{\varepsilon,\beta}\left[e^{-\tilde{\phi}}\mathcal{E}_0^{yz}+\mathcal{E}_1^{yz}
+\frac{1}{2}\tilde{n}^2(\partial_{yz}\varphi)^2\right]dx+\sum_{i=1}^5I_i^{yz}\leq C\delta\Lambda(t).
\end{align}
Recall the estimates of $I_i(i=1,...,5)$ in Lemma \ref{Clem3}, one has
\begin{align}\label{C81}
\sum_{i=1}^5I_i^{yz}\geq c\beta^3\Vert\partial_{yz}\vartheta\Vert_{\varepsilon-3,\beta}^2+c\beta\Vert\partial_{yz}(\nabla\varphi,{\rm div}\psi,\nabla\zeta,\nabla\sigma)\Vert_{\varepsilon-1,\beta}^2-C\delta\Lambda(t).
\end{align}

Combining these, we see that
\begin{align}\label{C82}
\frac{d}{dt}&\int_{\mathbb{R}_+^N}W_{\varepsilon,\beta}\left[e^{-\tilde{\phi}}\mathcal{E}_0^{yz}+\mathcal{E}_1^{yz}
+\frac{1}{2}\tilde{n}^2(\partial_{yz}\varphi)^2\right]dx\nonumber\\
&+\beta^3\Vert\partial_{yz}\vartheta\Vert_{\varepsilon-3,\beta}^2+\beta\Vert\partial_{yz}(\nabla\varphi,{\rm div}\psi,\nabla\zeta,\nabla\sigma)\Vert_{\varepsilon-1,\beta}^2\leq C\delta\Lambda(t).
\end{align}

Multiply \eqref{C82} by $(1+\beta\tau)^\xi$ and integrate over $(0,t)$ to give the desired estimate \eqref{C105}.
\end{proof}
\begin{lem}\label{Clem7}
Under the same assumption as in Proposition \ref{CProp1}, it holds for any $\xi\geq0$ and $t\in [0,T]$ that
\begin{align}\label{C106}
(1+&\beta t)^\xi\Vert\partial_{yz}\vartheta(t)\Vert_{\varepsilon,\beta,1}^2+\int_0^t(1+\beta\tau)^\xi\left\{\beta^3\Vert\partial_{yz}\vartheta(\tau)\Vert_{\varepsilon-3,\beta}^2
+\beta\Vert\partial_{yz}\nabla\vartheta(\tau)\Vert_{\varepsilon-1,\beta}^2\right\}d\tau\nonumber\\
\leq &C\Vert\partial_{yz}\vartheta(0)\Vert_{\varepsilon,\beta,1}^2+C\xi\int_0^t(1+\beta\tau)^{\xi-1}\beta\Vert\partial_{yz}\vartheta(\tau)\Vert_{\varepsilon,\beta,1}^2d\tau\nonumber\\
&+C\int_0^t(1+\beta\tau)^\xi\beta\Vert\partial_{yz}\nabla\sigma(\tau)\Vert_{\varepsilon-1,\beta}^2d\tau+C\int_0^t(1+\beta \tau)^\xi\int_{\mathbb{R}^{N-1}}(\partial_{yz}\sigma_{x_1})^2(\tau,0,x')dx'd\tau\nonumber\\
&+C\delta\int_0^t(1+\beta\tau)^\xi\Lambda (\tau)d\tau,
\end{align}
where positive constants $\delta$ and $C$ are independent of $T$.
\end{lem}
\begin{proof}
Apply the same argument as in Lemma \ref{Clem4} to $\eqref{C77}-\eqref{C78}$, with the help of Lemma \ref{Clem13}$-$Lemma \ref{Clem15}, it holds that
\begin{align}\label{C83}
\frac{d}{dt}\int_{\mathbb{R}_+^N}W_{\varepsilon,\beta}\left[e^{-\tilde{\phi}}\mathcal{E}_0^{yz}+\widetilde{\mathcal{E}}_1^{yz}
+\frac{1}{2}\tilde{n}^2(\partial_{yz}\varphi)^2\right]dx+I_1^{yz}+I_2^{yz}+I_5^{yz}+I_7^{yz}+I_8^{yz}\leq C\delta\Lambda(t).
\end{align}
Referring to the proof of Lemma \ref{Clem4}, we have
\begin{align}\label{C84}
I_1^{yz}+I_2^{yz}+I_5^{yz}+I_7^{yz}+I_8^{yz}\geq &c\beta^3\Vert\partial_{yz}\vartheta\Vert_{\varepsilon-3,\beta}^2+c\beta\Vert\nabla\partial_{yz}\vartheta\Vert_{\varepsilon-1,\beta}^2\nonumber\\
&-C\beta\Vert\nabla\partial_{yz}\sigma\Vert_{\varepsilon-1,\beta}^2-C\int_{\mathbb{R}^{N-1}}\sigma_{yzx_1}^2(t,0,x')dx'-C\delta\Lambda(t).
\end{align}
In light of \eqref{C83} and \eqref{C84}, we obtain
\begin{align}\label{C85}
 &\frac{d}{dt}\int_{\mathbb{R}_+^N}W_{\varepsilon,\beta}\left[e^{-\tilde{\phi}}\mathcal{E}_0^{yz}+\widetilde{\mathcal{E}}_1^{yz}
+\frac{1}{2}\tilde{n}^2(\partial_{yz}\varphi)^2\right]dx+\beta^3\Vert\partial_{yz}\vartheta\Vert_{\varepsilon-3,\beta}^2
+\beta\Vert\nabla\partial_{yz}\vartheta\Vert_{\varepsilon-1,\beta}^2\nonumber\\
\leq& C\beta\Vert\nabla\partial_{yz}\sigma\Vert_{\varepsilon-1,\beta}^2+C\int_{\mathbb{R}^{N-1}}(\partial_{yz}\sigma_{x_1})^2(t,0,x')dx'+C\delta\Lambda(t).
\end{align}

We multiply this inequality by $(1+\beta\tau)^\xi$ and integrate over $(0,t)$ to get \eqref{C106}.
\end{proof}
\begin{lem}\label{Clem8}
Under the same assumption as in Proposition \ref{CProp1}, it holds for any $\xi\geq0$ and $t\in [0,T]$ that
\begin{align}\label{C107}
(1+&\beta t)^\xi\Vert\partial_{yz}\vartheta(t)\Vert_{\varepsilon,\beta}^2+\int_0^t(1+\beta \tau)^\xi\int_{\mathbb{R}^{N-1}}(\partial_{yz}\sigma_{x_1})^2(\tau,0,x')dx'd\tau\nonumber\\
\leq &C\Vert\partial_{yz}\vartheta(0)\Vert_{\varepsilon,\beta}^2+C\Vert\vartheta_0\Vert_{\varepsilon,\beta,1}^2
+C\xi\int_0^t(1+\beta\tau)^{\xi-1}\beta\Vert\partial_{yz}\vartheta(\tau)\Vert_{\varepsilon,\beta}^2d\tau\nonumber\\
&+C\int_0^t(1+\beta\tau)^\xi\left\{\beta^3\Vert\partial_{yz}\vartheta(\tau)\Vert_{\varepsilon-3,\beta}^2+\beta\Vert\partial_{yz}(\nabla\varphi,{\rm div}\psi,\nabla\zeta,\nabla\sigma)(\tau)\Vert_{\varepsilon-1,\beta}^2\right\}d\tau\nonumber\\
&+C\delta\int_0^t(1+\beta\tau)^\xi\Lambda (\tau)d\tau,
\end{align}
where positive constants $\delta$ and $C$ are independent of $T$.
\end{lem}
\begin{proof}
Now we follow the proof of Lemma \ref{Clem5}. We apply $e^{-v}\partial_t$ to $\partial_{yz}\eqref{C15}_{4}$, then multiply the resultant equality by $\partial_{yz}\sigma$. Using $\eqref{C77}_1-\eqref{C77}_3$ and multiplying the result by $W_{\varepsilon,\beta}=(1+\beta x_1)^\varepsilon$, one has
\begin{align}\label{C86}
\frac{d}{dt}&\int_{\mathbb{R}_+^N}W_{\varepsilon,\beta}\left(\frac{1}{2}e^{-\phi-v}\lvert\partial_{yz}\sigma\rvert^2+\frac{1}{2}e^{-v}\lvert\nabla\partial_{yz}\sigma\rvert^2+
\frac{\mathcal{E}_0^{yz}}{\tilde{n}}\right)dx+I_9^{yz}+I_{10}^{yz}\leq C\delta\Lambda(t),
\end{align}
where we also have used  Lemma \ref{Clem13}$-$Lemma \ref{Clem15}.
Similar to \eqref{C66} and \eqref{C73}, one obtains
\begin{align}\label{C87}
I_9^{yz}+I_{10}^{yz}\geq c\int_{\mathbb{R}^{N-1}}\sigma_{yzx_1}^2(t,0,x')dx'-C\beta^3\Vert\partial_{yz}\vartheta\Vert_{\varepsilon-3,\beta}^2-C\beta\Vert\partial_{yz}(\nabla\varphi,{\rm div}\psi,\nabla\sigma)\Vert_{\varepsilon-1,\beta}^2-C\delta\Lambda(t).
\end{align}
Combining \eqref{C86}$-$\eqref{C87} gives
\begin{align}\label{C88}
\frac{d}{dt}&\int_{\mathbb{R}_+^N}W_{\varepsilon,\beta}\left(\frac{1}{2}e^{-\phi-v}\lvert\partial_{yz}\sigma\rvert^2+\frac{1}{2}e^{-v}\lvert\nabla\partial_{yz}\sigma\rvert^2+
\frac{\mathcal{E}_0^{yz}}{\tilde{n}}\right)dx+\int_{\mathbb{R}^{N-1}}\sigma_{yzx_1}^2(t,0,x')dx'\nonumber\\
&\leq C\beta^3\Vert\partial_{yz}\vartheta\Vert_{\varepsilon-3,\beta}^2+C\beta\Vert\partial_{yz}(\nabla\varphi,{\rm div}\psi,\nabla\sigma)\Vert_{\varepsilon-1,\beta}^2+C\delta\Lambda(t).
\end{align}

With the help of the elliptic estimate in Lemma \ref{Clem11} $\Vert\partial_{yz}(\sigma,\nabla\sigma)(0)\Vert_{\varepsilon,\beta}\leq C\Vert\varphi_0\Vert_{\varepsilon,\beta,1}$, the desired estimate \eqref{C107} is obtained by multiplying \eqref{C88} by $(1+\beta\tau)^\xi$ and integrating the resulting inequality over $(0,t)$.
\end{proof}

We now present the proof of the Proposition \ref{CProp1}.

For the case of $N=2,3$, by the same procedure as in deriving \eqref{C76} from Lemma \ref{Clem6}$-$Lemma \ref{Clem8}, one has
\begin{align}\label{C89}
(1+\beta t)^\xi&\Vert\partial_{yz}\vartheta(t)\Vert_{\varepsilon,\beta,1}^2+\int_0^t(1+\beta\tau)^\xi\left\{\beta^3\Vert\partial_{yz}\vartheta(\tau)
\Vert_{\varepsilon-3,\beta}^2+\beta
\Vert\partial_{yz}(\nabla\vartheta,\nabla\sigma)(\tau)\Vert_{\varepsilon-1,\beta}^2\right\}d\tau \nonumber\\
\leq &C\Vert\partial_{yz}\vartheta(0)\Vert_{\varepsilon,\beta,1}^2+C\Vert\vartheta(0)\Vert_{\varepsilon,\beta,1}^2
+C\xi\int_0^t(1+\beta\tau)^{\xi-1}\beta\Vert\partial_{yz}\vartheta(\tau)\Vert_{\varepsilon,\beta,1}^2d\tau \nonumber\\
&+C\delta\int_0^t(1+\beta t)^\xi\Lambda(\tau)d\tau.
\end{align}
We add the sum of \eqref{C89} for the entire combination of $(y,z)$ to \eqref{C76}, apply Lemma \ref{Clem13}$-$Lemma \ref{Clem14}, and then choose $\delta$ small enough, which derives
\begin{align}\label{C90}
(1+\beta t)^\xi&\Vert\vartheta(t)\Vert_{\varepsilon,\beta,3}^2+\int_0^t\beta^3(1+\beta\tau)^\xi\Vert\vartheta(\tau)\Vert_{\varepsilon-3,\beta,3}^2d\tau\nonumber\\
&\leq C\Vert\vartheta(0)\Vert_{\varepsilon,\beta,3}^2+C\xi\int_0^t(1+\beta\tau)^{\xi-1}\beta\Vert\vartheta(\tau)\Vert_{\varepsilon,\beta,3}^2d\tau.
\end{align}
By applying induction argument in \cite{KM1985} and \cite{MN1998} with $\xi=(\lambda-\varepsilon)/3+\kappa(\kappa>0)$ and the elliptic estimates Lemma \ref{Clem11} yields
\begin{align}\label{C91}
(1+\beta t)&^{(\lambda-\varepsilon)/3+\kappa}\left(\Vert\vartheta(t)\Vert_{\varepsilon,\beta,s}^2+\Vert\sigma(t)\Vert_{\varepsilon,\beta,s+2}^2\right)\nonumber\\
&+\int_0^t\beta^3(1+\beta \tau)^{(\lambda-\varepsilon)/3+\kappa}\left(\Vert\vartheta(\tau)\Vert_{\varepsilon-3,\beta,s}^2+\Vert\sigma(\tau)\Vert_{\varepsilon-3,\beta,s+2}^2\right)d\tau\leq C(1+\beta t)^\kappa\Vert\vartheta_0\Vert_{\lambda,\beta,s}^2.
\end{align}

For the case of $N=1$, by \eqref{C76} and \eqref{C125}, we obtain from Lemma\ref{Clem13} and Lemma\ref{Clem14} that
\begin{align}\label{C126}
(1+\beta t)^\xi&\Vert\vartheta(t)\Vert_{\varepsilon,\beta,2}^2+\int_0^t\beta^3(1+\beta\tau)^\xi\Vert\vartheta(\tau)\Vert_{\varepsilon-3,\beta,2}^2d\tau\nonumber\\
&\leq C\Vert\vartheta(0)\Vert_{\varepsilon,\beta,2}^2+C\xi\int_0^t(1+\beta\tau)^{\xi-1}\beta\Vert\vartheta(\tau)\Vert_{\varepsilon,\beta,2}^2d\tau,
\end{align}
which corresponds to \eqref{C90}. With the above estimate and \eqref{C91} in hand, we complete the proof of Proposition \ref{CProp1}.

\section{Energy estimates for the nondegenerate case}

In this section, we study the stability of the planar stationary solution to \eqref{C4} for the nondegenerate case \eqref{C92}. From the local existence result in Lemma \ref{Clem2} and a \emph{priori} estimates in Proposition \ref{CProp2}, we can obtain the global existence of solution by the standard continuity argument. Hence, we only focus on the proof of Proposition \ref{CProp2}. Due to the different properties of the stationary solution in \eqref{C11} and \eqref{C13}, the proof of the a \emph{priori} estimates for the nondegenerate problem is easier than that for the degenerate problem. We omit further details.

\subsection{A \emph{priori} estimates for the nondegenerate case}
The section is devoted to show a \emph{priori} estimates of Proposition \ref{CProp2}. To this end, we define the following notation for convenience:
\begin{align*}
\vartheta(t,x):=(\varphi,\psi,&\zeta)(t,x),\quad \vartheta_0(x):=(\varphi_0,\psi_0,\zeta_0)(x),\\
N_\lambda(T):&=\sup_{0\leq t\leq T}\Vert e^{\lambda x_1 /2}\vartheta(t)\Vert_{H^s}.
\end{align*}
\begin{prop}\label{CProp2}
Assume the same conditions on $N,m,T_\infty,u_\infty, \lambda_0$, and $\lambda$ hold as in Theorem \ref{CThm1}.
\item {\rm(i)}\ \ Let $(\vartheta,\sigma)(x,t)$ be a solution to \eqref{C15}$-$\eqref{C17} which satisfies
\begin{align*}
(e^{\lambda x_1 /2}\vartheta,e^{\lambda x_1 /2}\sigma)\in (\mathscr{X}_s^0([0,T]))^{N+2}\times\mathscr{X}_s^2([0,T]),
\end{align*}
over a time interval $[0,T]$ for  $T>0$.
Then there exist positive constants $\delta$ and $C$ independent of $T$ such that if all the following conditions
\begin{align*}
\alpha>0,\ \ \beta\in(0,\lambda],\ \ {\rm and}\ \ \beta+(\phi_b+N_{\lambda}(T)+\alpha)/\beta\leq \delta,
\end{align*}
are satisfied, then it holds for any $t\in [0,T]$ that
\begin{align}\label{C108}
\Vert e^{\beta x_1 /2}\vartheta(t)\Vert_{H^s}^2+\Vert e^{\beta x_1 /2}\sigma(t)\Vert_{H^{s+2}}^2\leq C\Vert e^{\beta x_1 /2}\vartheta_0\Vert_{H^s}^2e^{-\alpha t}.
\end{align}
\item {\rm(ii)}\ \ Let $(\vartheta,\sigma)(x,t)$ be a solution to \eqref{C15}$-$\eqref{C17}
over a time interval $[0,T]$ for  $T>0$.
Then there exist positive constants $\delta$ and $C$ independent of $T$ such that if all the conditions
\begin{align*}
((1+\beta x_1)^{\lambda/2}\vartheta,(1+\beta x_1)^{\lambda/2}\sigma)\in (\mathscr{X}_s^0([0,T]))^{N+2}\times\mathscr{X}_s^2([0,T])
\end{align*}
and
\begin{align}\label{C116}
\beta+(\phi_b+N_{\lambda,\beta}(T))/\beta\leq \delta,\ \ \beta>0,
\end{align}
are satisfied, then it holds for any $\varepsilon\in (0,\lambda]$ and $t\in [0,T]$ that
\begin{align}\label{C109}
\Vert\vartheta(t)\Vert_{\varepsilon,\beta,s}^2+\Vert\sigma(t)\Vert_{\varepsilon,\beta,s+2}^2\leq C\Vert\vartheta_0\Vert_{\lambda,\beta,s}^2(1+\beta t)^{-(\lambda-\varepsilon)}.
\end{align}
\end{prop}

The proof of Proposition \ref{CProp2} will be proved by following Lemma \ref{Clem9}$-$Lemma \ref{Clem10} at the end of this section.
\begin{lem}\label{Clem9}
 Under the same assumption as in Proposition \ref{CProp2}, the following estimates hold for any $\xi\geq0$ and $t\in [0,T]$ that
\begin{align}\label{C110}
(1+\beta t)^\xi&\Vert(\vartheta,\nabla\varphi,{\rm div} \psi,\nabla\zeta)(t)\Vert_{\varepsilon,\beta}^2+\int_0^t(1+\beta\tau)^\xi\beta\Vert(\vartheta,\nabla\varphi,{\rm div}\psi,\nabla\zeta,\nabla\sigma)(\tau)\Vert_{\varepsilon-1,\beta}^2d\tau\nonumber\\
\leq C\Vert\vartheta_0\Vert_{\varepsilon,\beta,1}^2&+C\xi\int_0^t(1+\beta\tau)^{\xi-1}\beta\Vert\vartheta(\tau)\Vert_{\varepsilon,\beta,1}^2d\tau+C\delta\int_0^t
(1+\beta\tau)^\xi\beta\Vert\nabla\psi(\tau)\Vert_{\varepsilon-1,\beta}^2d\tau,
\end{align}
\begin{align}\label{C111}
(1+&\beta t)^\xi\Vert\vartheta(t)\Vert_{\varepsilon,\beta,1}^2+\int_0^t(1+\beta\tau)^\xi\beta\Vert\nabla\vartheta(\tau)\Vert_{\varepsilon-1,\beta,1}^2d\tau\nonumber\\
\leq &C\Vert\vartheta_0\Vert_{\varepsilon,\beta,1}^2+C\xi\int_0^t(1+\beta\tau)^{\xi-1}\beta\Vert\vartheta(\tau)\Vert_{\varepsilon,\beta,1}^2d\tau\nonumber\\
&+C\int_0^t(1+\beta\tau)^\xi\beta\Vert\nabla\sigma(\tau)\Vert_{\varepsilon-1,\beta}^2d\tau+C\int_0^t(1+\beta \tau)^\xi\int_{\mathbb{R}^{N-1}}\sigma_{x_1}^2(\tau,0,x')dx'd\tau,
\end{align}
and
\begin{align}\label{C112}
(1+&\beta t)^\xi\Vert\vartheta(t)\Vert_{\varepsilon,\beta}^2+\int_0^t(1+\beta \tau)^\xi\int_{\mathbb{R}^{N-1}}\sigma_{x_1}^2(\tau,0,x')dx'd\tau\nonumber\\
\leq &C\Vert\vartheta_0\Vert_{\varepsilon,\beta,1}^2+C\xi\int_0^t(1+\beta\tau)^{\xi-1}\beta\Vert\vartheta(\tau)\Vert_{\varepsilon,\beta,1}^2d\tau+C\int_0^t(1+\beta\tau)^\xi\beta\Vert(\vartheta,\nabla\varphi,{\rm div}\psi,\nabla\zeta)(\tau)\Vert_{\varepsilon-1,\beta}^2d\tau,
\end{align}
where positive constants $\delta$ and $C$ are independent of $T$.
\end{lem}
\begin{proof}
As for the proof of \eqref{C110}, we need to reevaluate the terms $I_1$, $I_5$, and $\int_{\mathbb{R}_+^N}W_{\varepsilon,\beta}\mathcal{N}_1dx$ in \eqref{C31}.
For the estimate of $I_1$, it holds that
\begin{align}\label{C113}
I_1\geq \int_{\mathbb{R}_+^N}\varepsilon\beta W_{\varepsilon-1,\beta}&\left\{\frac{(1+RT_\infty)}{2}\lvert u_\infty\rvert\varphi^2-RT_\infty\varphi\psi_1+\frac{m}{2}\lvert
u_\infty\rvert\psi_1^2-R\psi_1\zeta+\frac{R\lvert u_\infty\rvert}{2(\gamma-1)T_\infty}\zeta^2+\psi_1\sigma\right\}dx\nonumber\\
&+c\beta\Vert\psi'\Vert_{\varepsilon-1,\beta}^2-C(N_{\lambda,\beta}(T)+\phi_b)\Vert(\vartheta,\sigma)\Vert_{\varepsilon-1,\beta}^2\nonumber\\
\geq \int_{\mathbb{R}_+^N}\varepsilon\beta W_{\varepsilon-1,\beta}&\left\{\frac{(1+RT_\infty)}{2}\lvert u_\infty\rvert\varphi^2-RT_\infty\varphi\psi_1+\frac{m\lvert u_\infty\rvert^2-1}{2\lvert u_\infty\rvert}\psi_1^2-R\psi_1\zeta+\frac{R\lvert u_\infty\rvert}{2(\gamma-1)T_\infty}\zeta^2\right.\nonumber\\
&\left.-\frac{\lvert u_\infty\rvert}{2}\sigma^2\right\}dx+c\beta\Vert\psi'\Vert_{\varepsilon-1,\beta}^2-C(N_{\lambda,\beta}(T)+\phi_b)\Vert(\vartheta,\sigma)\Vert_{\varepsilon-1,\beta}^2,
\end{align}
where $\psi':=(\psi_2,...,\psi_N)$ and the derivation of the second inequality utilizes the Cauchy-Schwarz inequality
$\sigma\psi_1\geq-\left(\frac{\lvert u_\infty \rvert}{2}\lvert\sigma\rvert^2+\frac{1}{2\lvert u_\infty \rvert}\psi_1^2\right)$.

Now we need to deal with the term $-\frac{\lvert u_\infty\rvert}{2}\int_{\mathbb{R}_+^N}\varepsilon\beta W_{\varepsilon-1,\beta} \sigma^2dx$.
Precisely, we multiply $\eqref{C15}_4$ by $-\varepsilon\beta\sigma W_{\varepsilon-1,\beta}$ and integrate the result over $\mathbb{R}_+^N$ to get
\begin{align*}
&\int_{\mathbb{R}_+^N} \varepsilon\beta W_{\varepsilon-1,\beta}\lvert\nabla\sigma\rvert^2dx+\int_{\mathbb{R}_+^N}\frac{1}{2} \varepsilon(\varepsilon-1)(\varepsilon-2)\beta^3 W_{\varepsilon-3,\beta}\sigma^2dx\nonumber\\
\leq &-\int_{\mathbb{R}_+^N}\varepsilon\beta W_{\varepsilon-1,\beta}(\varphi\sigma+\sigma^2)dx+C(N_{\lambda,\beta}(T)+\phi_b)\int_{\mathbb{R}_+^N}\varepsilon\beta
W_{\varepsilon-1,\beta}\varphi^2dx\nonumber\\
\leq &\int_{\mathbb{R}_+^N}\frac{1}{2}\varepsilon\beta W_{\varepsilon-1,\beta}\varphi^2dx-\int_{\mathbb{R}_+^N}\frac{1}{2}\varepsilon\beta W_{\varepsilon-1,\beta}\sigma^2dx+C(N_{\lambda,\beta}(T)+\phi_b)\int_{\mathbb{R}_+^N}\varepsilon\beta
W_{\varepsilon-1,\beta}\varphi^2dx,
\end{align*}
then
\begin{align}\label{C114}
&\int_{\mathbb{R}_+^N}\frac{\lvert u_\infty\rvert}{2}\varepsilon\beta W_{\varepsilon-1,\beta}\varphi^2dx-\int_{\mathbb{R}_+^N}\frac{\lvert u_\infty\rvert}{2}\varepsilon\beta W_{\varepsilon-1,\beta}\sigma^2dx\nonumber\\
\geq &\int_{\mathbb{R}_+^N}\lvert u_\infty\rvert\varepsilon\beta W_{\varepsilon-1,\beta}\lvert\nabla\sigma\rvert^2dx-C(N_{\lambda,\beta}(T)+\phi_b+\beta^3)\Vert\varphi\Vert_{\varepsilon-1,\beta}^2,
\end{align}
where we have used the elliptic estimate in Lemma \ref{Clem11}.
Substituting \eqref{C114} into \eqref{C113}, using \eqref{C92} and \eqref{C116} gives that
\begin{align}\label{C115}
I_1\geq &\int_{\mathbb{R}_+^N}\varepsilon\beta W_{\varepsilon-1,\beta}\left\{\frac{RT_\infty}{2}\lvert u_\infty\rvert\varphi^2-RT_\infty\varphi\psi_1+\frac{m\lvert u_\infty\rvert^2-1}{2\lvert u_\infty\rvert}\psi_1^2-R\psi_1\zeta+\frac{R\lvert u_\infty\rvert}{2(\gamma-1)T_\infty}\zeta^2\right\}dx\nonumber\\
&+c\beta\Vert(\psi',\nabla\sigma)\Vert_{\varepsilon-1,\beta}^2
-C(N_{\lambda,\beta}(T)+\phi_b+\beta^3)\Vert(\vartheta,\sigma)\Vert_{\varepsilon-1,\beta}^2\nonumber\\
\geq &(c-C\delta)\beta\Vert(\vartheta, \nabla\sigma)\Vert_{\varepsilon-1,\beta}^2.
\end{align}

From \eqref{C11}, \eqref{C15}, \eqref{C116}, $\lambda\geq 2$,  the Sobolev inequality, the Cauchy-Schwarz inequality and the elliptic estimate in Lemma \ref{Clem11}, we can easily get
\begin{align}\label{C117}
\lvert I_5\rvert+\left\lvert\int_{\mathbb{R}_+^N}W_{\varepsilon,\beta}\mathcal{N}_1dx\right\rvert\leq C(N_{\lambda,\beta}+\phi_b)\beta\Vert\vartheta\Vert_{\varepsilon-1,\beta,1}^2
\leq C\beta\delta\Vert\vartheta\Vert_{\varepsilon-1,\beta,1}^2.
\end{align}
With the estimates \eqref{C115}$-$\eqref{C117} and \eqref{C33}$-$\eqref{C35} in hand, we can immediately get
\begin{align}\label{C118}
\frac{d}{dt}\int_{\mathbb{R}_+^N}W_{\varepsilon,\beta}\left(e^{-\tilde{\phi}}\mathcal{E}_0+\mathcal{E}_1+\frac{1}{2}\tilde{n}^2\varphi^2\right)dx
+\beta\Vert(\vartheta,\nabla\varphi, {\rm div}\psi, \nabla\zeta, \nabla\sigma)\Vert_{\varepsilon-1,\beta}^2
\leq C\delta\beta\Vert\nabla\psi\Vert_{\varepsilon-1,\beta}^2,
\end{align}
provided that $\delta>0$ is sufficiently small, where $\mathcal{E}_0$ and $\mathcal{E}_1$ are defined in \eqref{C20} and \eqref{C25} respectively. Multiply \eqref{C118} by $(1+\beta\tau)^\xi$ and integrate over $(0,t)$ to get \eqref{C110}. The estimates \eqref{C111} and \eqref{C112} have the same derivation as \eqref{C103} and \eqref{C104}, we omit the detials.
\end{proof}

Applying the same computational arguments on \eqref{C110}$-$\eqref{C112} used in \eqref{C76}, one has
\begin{align}\label{C122}
(1+\beta t)^\xi&\Vert\vartheta(t)\Vert_{\varepsilon,\beta,1}^2+\int_0^t(1+\beta\tau)^\xi\beta
\Vert(\vartheta,\nabla\vartheta,\nabla\sigma)(\tau)\Vert_{\varepsilon-1,\beta}^2d\tau \nonumber\\
&\leq C\Vert\vartheta(0)\Vert_{\varepsilon,\beta,1}^2+C\xi\int_0^t(1+\beta\tau)^{\xi-1}\beta\Vert\vartheta(\tau)\Vert_{\varepsilon,\beta,1}^2d\tau.
\end{align}
\begin{lem}\label{Clem10}
 Under the same assumption as in Proposition \ref{CProp2}, the following estimates hold for any $\xi\geq0$ and $t\in [0,T]$ that
\begin{align}\label{C119}
(1&+\beta t)^\xi\Vert\partial_{yz}(\vartheta,\nabla\varphi,{\rm div} \psi,\nabla\zeta)(t)\Vert_{\varepsilon,\beta}^2+\int_0^t(1+\beta\tau)^\xi\beta\Vert\partial_{yz}(\vartheta,\nabla\varphi,{\rm div}\psi,\nabla\zeta,\nabla\sigma)(\tau)\Vert_{\varepsilon-1,\beta}^2d\tau\nonumber\\
\leq &C\Vert\partial_{yz}(\vartheta,\nabla\varphi,{\rm div} \psi,\nabla\zeta)(0)\Vert_{\varepsilon,\beta}^2+C\xi\int_0^t(1+\beta\tau)^{\xi-1}\beta\Vert\partial_{yz}(\vartheta,\nabla\varphi,{\rm div}\psi,\nabla\zeta)(\tau)\Vert_{\varepsilon,\beta}^2d\tau\nonumber\\
&+C\delta\int_0^t(1+\beta\tau)^\xi \beta \Vert \vartheta(\tau)\Vert_{\varepsilon-1,\beta,3}^2d\tau,
\end{align}

\begin{align}\label{C120}
(1+&\beta t)^\xi\Vert\partial_{yz}\vartheta(t)\Vert_{\varepsilon,\beta,1}^2+\int_0^t(1+\beta\tau)^\xi\beta\Vert\partial_{yz}\vartheta(\tau)\Vert_{\varepsilon-1,\beta,1}^2d\tau\nonumber\\
\leq &C\Vert\partial_{yz}\vartheta(0)\Vert_{\varepsilon,\beta,1}^2+C\xi\int_0^t(1+\beta\tau)^{\xi-1}\beta\Vert\partial_{yz}\vartheta(\tau)\Vert_{\varepsilon,\beta,1}^2d\tau\nonumber\\
&+C\int_0^t(1+\beta\tau)^\xi\beta\Vert\partial_{yz}\nabla\sigma(\tau)\Vert_{\varepsilon-1,\beta}^2d\tau+C\int_0^t(1+\beta \tau)^\xi\int_{\mathbb{R}^{N-1}}(\partial_{yz}\sigma_{x_1})^2(\tau,0,x')dx'd\tau\nonumber\\
&+C\delta\int_0^t(1+\beta\tau)^\xi \beta \Vert \vartheta(\tau)\Vert_{\varepsilon-1,\beta,3}^2d\tau,
\end{align}
and
\begin{align}\label{C121}
(1+&\beta t)^\xi\Vert\partial_{yz}\vartheta(t)\Vert_{\varepsilon,\beta}^2+\int_0^t(1+\beta \tau)^\xi\int_{\mathbb{R}^{N-1}}(\partial_{yz}\sigma_{x_1})^2(\tau,0,x')dx'd\tau\nonumber\\
\leq &C\Vert\partial_{yz}\vartheta(0)\Vert_{\varepsilon,\beta}^2+C\Vert\vartheta_0\Vert_{\varepsilon,\beta,1}^2
+C\xi\int_0^t(1+\beta\tau)^{\xi-1}\beta\Vert\partial_{yz}\vartheta(\tau)\Vert_{\varepsilon,\beta}^2d\tau\nonumber\\
&+C\int_0^t(1+\beta\tau)^\xi\beta\Vert\partial_{yz}(\vartheta,\nabla\varphi,{\rm div}\psi,\nabla\zeta,\nabla\sigma)(\tau)\Vert_{\varepsilon-1,\beta}^2d\tau\nonumber\\
&+C\delta\int_0^t(1+\beta\tau)^\xi \beta \Vert \vartheta(\tau)\Vert_{\varepsilon-1,\beta,3}^2d\tau,
\end{align}
where positive constants $\delta$ and $C$ are independent of $T$.
\end{lem}
\begin{proof}
Similarly, we can refer to the Lemma \ref{Clem6}$-$Lemma \ref{Clem8} to prove the estimates \eqref{C119}$-$\eqref{C121} respectively.
\end{proof}
We now present the proof of the Proposition \ref{CProp2}(ii).

For the case of $N=2,3$, by the same procedure as in deriving \eqref{C89}, using \eqref{C119}$-$\eqref{C121}, we get
\begin{align}\label{C123}
(1+\beta t)^\xi&\Vert\partial_{yz}\vartheta(t)\Vert_{\varepsilon,\beta,1}^2+\int_0^t(1+\beta\tau)^\xi\beta
\Vert\partial_{yz}(\vartheta,\nabla\vartheta,\nabla\sigma)(\tau)\Vert_{\varepsilon-1,\beta}^2d\tau \nonumber\\
\leq &C\Vert\partial_{yz}\vartheta(0)\Vert_{\varepsilon,\beta,1}^2+C\Vert\vartheta(0)\Vert_{\varepsilon,\beta,1}^2
+C\xi\int_0^t(1+\beta\tau)^{\xi-1}\beta\Vert\partial_{yz}\vartheta(\tau)\Vert_{\varepsilon,\beta,1}^2d\tau \nonumber\\
&+C\delta\int_0^t(1+\beta\tau)^\xi \beta \Vert \vartheta(\tau)\Vert_{\varepsilon-1,\beta,3}^2d\tau.
\end{align}
It follows from \eqref{C122} and \eqref{C123}, taking $\delta>0$ sufficiently small, that
\begin{align}\label{C124}
(1+\beta t)^\xi&\Vert\vartheta(t)\Vert_{\varepsilon,\beta,3}^2+\int_0^t\beta(1+\beta\tau)^\xi\Vert\vartheta(\tau)\Vert_{\varepsilon-1,\beta,3}^2d\tau\nonumber\\
&\leq C\Vert\vartheta(0)\Vert_{\varepsilon,\beta,3}^2+C\xi\int_0^t(1+\beta\tau)^{\xi-1}\beta\Vert\vartheta(\tau)\Vert_{\varepsilon,\beta,3}^2d\tau.
\end{align}

For the case of $N=1$, as in \eqref{C125}, we have
\begin{align}\label{C127}
(1+&\beta t)^\xi\Vert\vartheta_t\Vert_{\varepsilon,\beta,1}^2+\int_0^t(1+\beta\tau)^\xi\beta\Vert(\vartheta_t,\vartheta_{xt},\sigma_{xt})(\tau)\Vert_{\varepsilon-1,\beta}^2d\tau\nonumber\\
\leq &C\Vert\vartheta_t(0)\Vert_{\varepsilon,\beta,1}^2+C\xi\int_0^t(1+\beta\tau)^{\xi-1}\beta\Vert\vartheta_t\Vert_{\varepsilon,\beta,1}^2d\tau+C\delta\int_0^t(1+\beta\tau)^\xi\beta\Vert\vartheta(\tau)\Vert_{\varepsilon-1,\beta,2}^2d\tau.
\end{align}
Following the derivation of \eqref{C126}, one has
\begin{align}\label{C128}
(1+\beta t)^\xi&\Vert\vartheta(t)\Vert_{\varepsilon,\beta,2}^2+\int_0^t(1+\beta\tau)^\xi\beta\Vert\vartheta(\tau)\Vert_{\varepsilon-1,\beta,2}^2d\tau\nonumber\\
&\leq C\Vert\vartheta(0)\Vert_{\varepsilon,\beta,2}^2+C\xi\int_0^t(1+\beta\tau)^{\xi-1}\beta\Vert\vartheta(\tau)\Vert_{\varepsilon,\beta,2}^2d\tau.
\end{align}

By applying the same induction argument on \eqref{C124} and \eqref{C128} as in \cite{KM1985} and \cite{MN1998} with $\xi=(\lambda-\varepsilon)/3+\kappa(\kappa>0)$ and the elliptic estimates in Lemma \ref{Clem11} yields
\begin{align}\label{C129}
(1+\beta t)&^{\lambda-\varepsilon+\kappa}\left(\Vert\vartheta(t)\Vert_{\varepsilon,\beta,s}^2+\Vert\sigma(t)\Vert_{\varepsilon,\beta,s+2}^2\right)\nonumber\\
&+\int_0^t(1+\beta \tau)^{\lambda-\varepsilon+\kappa}\beta\left(\Vert\vartheta(\tau)\Vert_{\varepsilon,\beta,s}^2+\Vert\sigma(\tau)\Vert_{\varepsilon,\beta,s+2}^2\right)d\tau\leq C(1+\beta t)^\kappa\Vert\vartheta_0\Vert_{\lambda,\beta,s}^2.
\end{align}
Hence, we complete the proof of Proposition \ref{CProp2}.

\section*{Appendix A.}
In this appendix, we will give some basic results used in the proofs of Proposition \ref{CProp1} and Proposition \ref{CProp2}. The lemmas below are similar to ones obtained in \cite{NOS2012}.
\appendix
\begin{lem}\label{Clem11}
Assume $\sigma$ and $\varphi$ satisfy $\eqref{C15}_4$ and $s=[N/2]+2$.
\item {\rm(i)}\ \ Let $((1+\mu x_1)^{\lambda/2}\varphi,(1+\mu x_1)^{\lambda/2}\sigma)\in \mathscr{X}_s^0([0,T])\times\mathscr{X}_s^2([0,T])$ for positive constants
$\lambda$ and $\mu$. Then, for any constant $c_0\in(0,2]$, there exist positive constants $\delta$ and $C$ independent of $T$ such that if all the conditions $\alpha\leq\lambda$, $\beta\in(0,\mu]$, $\lvert\alpha\beta\rvert\leq c_0$, and $\lvert\phi_b\rvert+N_{\lambda,\beta}(T)\leq \delta$ are satisfied, then $\sigma$ satisfies $(1+\beta x_1)^{\alpha/2}\sigma\in \mathscr{X}_s^2([0,T])$ with
\begin{align}\label{C130}
\Vert(1+\beta x_1)^{\alpha/2}\partial_t^i\sigma\Vert_{H^j}\leq C\Vert(1+\beta x_1)^{\alpha/2}\varphi\Vert_{H^{i+j-2}},\ \ i\in\mathbb{Z}\cap[0,2],\ \ j\in\mathbb{Z}\cap[2,4-i].
\end{align}
\item {\rm(ii)}\ \ Let $(e^{\lambda x_1/2}\varphi,e^{\lambda x_1/2}\sigma)\in \mathscr{X}_s^0([0,T])\times\mathscr{X}_s^2([0,T])$ for positive constants
$\lambda$. Then, for any constant $c_0\in(0,\sqrt{2}]$, there exist positive constants $\delta$ and $C$ independent of $T$ such that if the conditions
$\beta\in(0,c_0]$ and $\lvert\phi_b\rvert+N_{\lambda,\beta}(T)\leq \delta$ are satisfied, then $e^{\beta x_1/2}\sigma$ satisfies $\sigma\in \mathscr{X}_s^2([0,T])$ with
\begin{align}\label{C131}
\Vert e^{\beta x_1/2}\partial_t^i\sigma\Vert_{H^j}\leq C\Vert e^{\beta x_1/2}{\alpha/2}\varphi\Vert_{H^{i+j-2}},\ \ i\in\mathbb{Z}\cap[0,2],\ \ j\in\mathbb{Z}\cap[2,4-i].
\end{align}
\end{lem}
\begin{proof}
The above estimates can be derived by the standard elliptic estimate on $\eqref{C15}_4$. For brevity, we omit their proofs.
\end{proof}

\begin{lem}\label{Clem12}
Under the same assumptions as in either Proposition \ref{CProp1} for the degenerate case or Proposition \ref{CProp2}(ii) for the nondegenerate case, it holds for any $t\in[0,T]$ and $\alpha\leq\lambda/2$ that
\begin{align}\label{C132}
\Vert ((1+\beta x_1)^\alpha\vartheta,(1+\beta x_1)^\alpha\nabla\vartheta)(t)\Vert_{L^\infty(\mathbb{R}_+^N)}\leq CN_{\lambda,\beta}(T),
\end{align}
\begin{align}\label{C133}
\Vert (1+\beta x_1)^\alpha\vartheta_t(t)\Vert_{L^\infty(\mathbb{R}_+^N)}\leq CN_{\lambda,\beta}(T).
\end{align}
\end{lem}
\begin{proof}
By applying the Sobolev inequality, we have
\begin{align*}
\Vert &((1+\beta x_1)^\alpha\vartheta,(1+\beta x_1)^\alpha\nabla\vartheta)(t)\Vert_{L^\infty(\mathbb{R}_+^N)}\\
&\leq C\Vert ((1+\beta x_1)^\alpha\vartheta,(1+\beta x_1)^\alpha\nabla\vartheta)(t)\Vert_{H^{s-1}(\mathbb{R}_+^N)}\\
&\leq C\Vert ((1+\beta x_1)^\alpha\vartheta(t)\Vert_{H^{s}(\mathbb{R}_+^N)}\\
&\leq CN_{\lambda,\beta}(T),
\end{align*}
which give the estimate of \eqref{C132}. The estimate \eqref{C133} immediately follows from \eqref{C132} owing to \eqref{C15}, Lemma \ref{Clem11}(i), and \eqref{C11} for nondegenerate case or \eqref{C13} for degenerate case.
\end{proof}

\begin{lem}\label{Clem13}
For the nondegenerate case, we assume the same conditions as in Proposition \ref{CProp2}(ii) and let $\delta$ suitably small. Then it holds for $\nu=\varepsilon-1$ or $\varepsilon$ that
 \begin{align}\label{C134}
\Vert \partial_t^i\vartheta(t)\Vert_{\nu,\beta,j}\leq C\Vert\vartheta(t)\Vert_{\nu,\beta,i+j},\ \ {\rm for} \ \ (i,j)\in \{(i.j)\in \mathbb{Z}^2|i,j\geq0, i+j\leq[N/2]+2\}.
\end{align}
For the degenerate case, we assume the same conditions as in Proposition \ref{CProp1} and let $\delta$ be suitably small. Then it holds for $\nu=\varepsilon-3,\varepsilon-1$ or $\varepsilon$ that
 \begin{align}\label{C135}
\Vert\vartheta_t(t)\Vert_{\nu,\beta,j}&\leq C\Vert(\nabla\vartheta,\nabla\sigma)(t)\Vert_{\nu,\beta}+C\beta\Vert\vartheta(t)\Vert_{\nu-2,\beta}\leq C\Vert\vartheta(t)\Vert_{\nu,\beta,1},\ \ N=1,2,3,\nonumber\\
\Vert(\nabla\vartheta_t,\vartheta_{tt})(t)\Vert_{\nu,\beta}&\leq\Vert\vartheta(t)\Vert_{\nu,\beta,2}, \ \ N=1,2,3,\nonumber\\
\Vert(\nabla^2\vartheta_t,\nabla\vartheta_{tt},\vartheta_{ttt})(t)\Vert_{\nu,\beta}&\leq C\Vert(\nabla\vartheta,\nabla^3\vartheta,\nabla\sigma,\nabla\sigma_{tt})(t)\Vert_{\nu,\beta}
+C\beta\Vert(\vartheta,\nabla^2\vartheta)(t)\Vert_{\nu-2,\beta}\nonumber\\
&\leq C\Vert\vartheta(t)\Vert_{\nu,\beta,3},\ \ N=2,3.
\end{align}
\end{lem}
\begin{proof}
These estimates are derived by the governing equations \eqref{C18} as well as the time-derivative system with the help of Lemma \ref{Clem11}(i).
\end{proof}

\begin{lem}\label{Clem14}
For the nondegenerate case, we assume the same conditions as in Proposition \ref{CProp2}(ii) and let $\delta$ suitably small. Then it holds for $\nu=\varepsilon-1,\varepsilon$ that
\begin{align*}
\Vert\partial_{x_1}\vartheta(t)\Vert_{\nu,\beta}&\leq C\Vert((\partial_t,\nabla')\vartheta,\vartheta)(t)\Vert_{\nu,\beta},\ \ \ \ N=1,2,3,\\
\Vert\nabla\partial_{x_1}\vartheta(t)\Vert_{\nu,\beta}&\leq C\Vert((\partial_t,\nabla')^2\vartheta,(\partial_t,\nabla')\vartheta,\vartheta)(t)\Vert_{\nu,\beta},\ \ \ \ N=1,2,3,\\
\Vert\nabla\partial_{x_1}^2\vartheta(t)\Vert_{\nu,\beta}&\leq C\Vert((\partial_t,\nabla')^3\vartheta,(\partial_t,\nabla')^2\vartheta,(\partial_t,\nabla')\vartheta,\vartheta)(t)\Vert_{\nu,\beta},\ \ \ \ N=2,3,
\end{align*}
where $\nabla'$ denotes a derivative with respect to the spatial variable other than $x_1$.

For the degenerate case, we assume the same conditions as in Proposition \ref{CProp1} and let $\delta$ suitably small. Then it holds for $\nu=\varepsilon-3,\varepsilon-1,\varepsilon$ that
\begin{align*}
\Vert\partial_{x_1}\vartheta(t)\Vert_{\nu,\beta}&\leq C\Vert((\partial_t,\nabla')\vartheta,\nabla\sigma)(t)\Vert_{\nu,\beta}+C\beta\Vert\vartheta(t)\Vert_{\nu-2,\beta}\\
&\leq C\Vert((\partial_t,\nabla')\vartheta,\vartheta)(t)\Vert_{\nu,\beta},\ \ \ \ N=1,2,3,\\
\Vert\nabla\partial_{x_1}\vartheta(t)\Vert_{\nu,\beta}&\leq C\Vert((\partial_t,\nabla')^2\vartheta,\nabla\vartheta,\vartheta)(t)\Vert_{\nu,\beta}\\
&\leq C\Vert((\partial_t,\nabla')^2\vartheta,(\partial_t,\nabla')\vartheta,\vartheta)(t)\Vert_{\nu,\beta},\ \ \ \ N=1,2,3,\\
\Vert\nabla\partial_{x_1}^2\vartheta(t)\Vert_{\nu,\beta}&\leq C\Vert((\partial_t,\nabla')^2\nabla\vartheta,\nabla\vartheta,\nabla\sigma)(t)\Vert_{\nu,\beta}+C\beta\Vert((\partial_t,\nabla')^2\vartheta,\vartheta)(t)\Vert_{\nu-2,\beta}\\
&\leq C\Vert((\partial_t,\nabla')^3\vartheta,(\partial_t,\nabla')^2\vartheta,(\partial_t,\nabla')\vartheta,\vartheta)(t)\Vert_{\nu,\beta}
,\ \ \ \ N=2,3,
\end{align*}
\end{lem}
\begin{proof}
It is proved by the similar argument as in \cite{DYZ2021,NOS2012}, we omit the details.
\end{proof}
\begin{lem}\label{Clem15}
Under the same assumptions as in Proposition \ref{CProp1}, for the case of $N=2,3$, it holds that
\begin{align}\label{C136}
\left\lvert\int W_{\alpha,\beta}f\partial^{(2)}g\partial^{(2)}hdx\right\rvert\leq CN_{\lambda,\beta}(T)(\Vert f\Vert_{\alpha-1,\beta}^2+\Vert\vartheta\Vert_{\alpha-3.\beta,3}^2),
\end{align}
where $\partial^{(2)}g$ and $\partial^{(2)}h$ denote the functions obtained by differentiating $\varphi$ or $\psi_1,...,\psi_N$ twice by any coordinates and $f$ is an any function which satisfies $(1+\beta x_1)^{(\alpha-1)/2}f\in L^2$.
\end{lem}
\begin{proof}
By using Cauchy-Schwartz inequality, Sobolev inequality and $H\ddot{o}lder$ inequality, we get
\begin{align*}
&\left\lvert \int W_{\alpha,\beta}f\partial^{(2)}g\partial^{(2)}hdx\right\rvert\\
&\leq \Vert f\Vert_{\alpha-1,\beta}\Vert W_{2,\beta}\partial^{(2)}g\Vert_{L^4}\Vert W_{(\alpha-3)/2,\beta}\partial^{(2)}h\Vert_{L^4}\\
&\leq C\Vert f\Vert_{\alpha-1,\beta}\Vert\partial^{(2)}g\Vert_{4,\beta,1}\Vert\partial^{(2)}h\Vert_{\alpha-3,\beta,1}\\
&\leq C\Vert\partial^{(2)}g\Vert_{\lambda,\beta,1}(\Vert\partial^{(2)}h\Vert_{\alpha-3,\beta,1}^2+\vert f\Vert_{\alpha-1,\beta}^2)\\
&\leq CN_{\lambda,\beta}(T)(\Vert\vartheta\Vert_{\alpha-3,\beta,3}^2+\Vert f\Vert_{\alpha-1,\beta}^2).
\end{align*}
\end{proof}

\section*{Acknowledgement}
Lei Yao's research  was partially supported by National Natural Science Foundation of China
$\#$12171390, $\#$11931013, and the Fundamental Research Funds for the Central Universities under Grant: G2023KY05102.
Haiyan Yin's  research  was partially supported by National Natural Science Foundation of China $\#$12071163, and
the Natural Science Foundation of Fujian Province of China $\#$2021J01305, $\#$2020J01071.

\end{document}